\documentclass[11pt]{amsart}
\usepackage{amsfonts}

\newtheorem{theorem}{Theorem}
\theoremstyle{plain}

\newtheorem{lemma}{Lemma}

\newtheorem{remark}{Remark}

\numberwithin{equation}{section}
\input{tcilatex}

\begin{document}
\title{Multiple solutions to singular fourth order elliptic equations }
\author{Mohammed Benalili and Kamel Tahri}
\address{Dept. Maths. Faculty of Sciences, University UABB, B.P. 119 Tlemcen}
\email{m\_benalili@mail.univ-tlemcen.dz}
\maketitle

\begin{abstract}
Using the method of Nehari manifold, we prove the existence of at least two
distinct weak solutions to elliptic equation of four order with
singulatities and with critical Sobolev growth.
\end{abstract}

\section{\protect\smallskip Introduction}

Fourth order elliptic equations have been intensively investigated the last
tree decades particularly after the discovery of an important conformally
invariant operator by Paneitz on $4$ - dimensional Riemannian manifolds \cite%
{19} and whose definition was extended to higher dimension by Branson \cite%
{8}.This operator is closely related to the problem of prescribed $Q$-
curvature. Many works have been devoted to this subject ( see \cite{1}, \cite%
{2}, \cite{3}, \cite{4}, \cite{5}, \cite{6}, \cite{7}, \cite{11}, \cite{12}, 
\cite{13}, \cite{14}, \cite{15}, \cite{16}, \cite{17}, \cite{20}, \cite{21}, 
\cite{22}, \cite{23}, \cite{24} ). Let $(M,g)$ a compact smooth Riemannian
of dimension $n\geq 5$ with a metric $g$. We denote by $H_{2}^{2}(M)$ the
standard Sobolev space which is the completion of the space $C^{\infty
}\left( M\right) $ with respect to the norm 
\begin{equation*}
\left\Vert \varphi \right\Vert _{2,2}=\sum_{k=0}^{k=2}\left\Vert \nabla
^{k}\varphi \right\Vert _{2}\text{.}
\end{equation*}%
$H_{2}^{2}(M)$ will be endowed with the equivalent suitable norm

\begin{equation*}
\left\Vert u\right\Vert _{H_{2}^{2}(M)}=(\int_{M}\left( \left( \Delta
_{g}u\right) ^{2}+\left\vert \nabla _{g}u\right\vert ^{2}+u^{2}\right)
dv_{g})^{\frac{1}{2}}\text{.}
\end{equation*}%
Recently, Madani \cite{18}, has considered the Yamabe problem with
singularities which he solved under some geometric conditions. The first
author in \cite{6} considered singular fourth order elliptic equations with
of the form%
\begin{equation}
\Delta ^{2}u-\nabla ^{i}\left( a(x)\nabla _{i}u\right) +b(x)u=f\left\vert
u\right\vert ^{N-2}u  \label{1}
\end{equation}%
where the functions $a$ and $b$ are in $L^{s}(M)$, $s>\frac{n}{2}$ and in $%
L^{p}(M)$, $p>\frac{n}{4}$ respectively, $N=\frac{2n}{n-4}$ is the Sobolev
critical exponent in the embedding $H_{2}^{2}\left( M\right) \hookrightarrow
L^{N}\left( M\right) $. He established the following result. Let $\left(
M,g\right) $ be a compact $n$-dimensional Riemannian manifold, $n\geq 6$, $%
a\in L^{s}(M)$, $b\in L^{p}(M)$, with $s>\frac{n}{2}$, $p>\frac{n}{4}$, $f$ $%
\in C^{\infty }(M)$ a positive function and $x_{o}\in M$ such that $%
f(x_{o})=\max_{x\in M}f(x)$.

\begin{theorem}
\label{th01} Let $\left( M,g\right) $ be a compact $n$-dimensional
Riemannian manifold, $n\geq 6$, $a\in L^{s}(M)$, $b\in L^{p}(M)$, with $s>%
\frac{n}{2}$, $p>\frac{n}{4}$, $f$ $\in C^{\infty }(M)$ a positive function
and $P\in M$ such that $f(P)=\max_{x\in M}f(x)$.

For $n\geq 10$,or $n=9$ and $\frac{9}{4}<p<11$ or $n=8$ and $2<p<5$ or $n=7$
and $\frac{7}{2}<s<9$ , $\frac{7}{4}<p<3$ we suppose that 
\begin{equation*}
\frac{n^{2}+4n-20}{6\left( n-6\right) (n^{2}-4)}R_{g}\left( P\right) -\frac{%
n-4}{2n\left( n-2\right) }\frac{\Delta f(P)}{f(P)}>0\text{.}
\end{equation*}%
For $n=6$ and $\frac{3}{2}<p<2$, $3<s<4$, we suppose that 
\begin{equation*}
R_{g}(P)>0\text{.}
\end{equation*}%
Then the equation (\ref{1}) has a non trivial weak solution $u$ in $%
H_{2}^{2}\left( M\right) $. Moreover if $a\in H_{1}^{s}\left( M\right) $,
then

$u\in $ $C^{0,\beta }$, for some $\beta \in \left( 0,1-\frac{n}{4p}\right) $%
..
\end{theorem}

For fixed $R\in M$, we define the function $\rho $ on $M$ by

\begin{equation}
\rho (Q)=\left\{ 
\begin{array}{c}
d(R,Q)\text{ \ if \ \ \ \ \ \ }d(R,Q)<\delta (M) \\ 
\delta (M)\text{ \ if\ \ \ \ \ }d(R,Q)\geq \delta (M)%
\end{array}%
\right.  \label{2'}
\end{equation}%
where $\delta (M)$ denotes the injectivity radius of $M$.

In this paper, we are concerned with the following problem: for real numbers 
$\sigma $ and $\mu $, consider the equation in the distribution sense

\begin{equation}
\Delta ^{2}u-\nabla ^{i}(a\rho ^{-\mu }\nabla _{i}u)+\rho ^{-\alpha
}bu=\lambda \left\vert u\right\vert ^{q-2}u+f(x)\left\vert u\right\vert
^{N-2}u  \label{3'}
\end{equation}%
where the functions $a$ and $b$ are smooth $M$ and $1<q<2$. \ Denote also by 
$P_{g}$ the operator define on $H_{2}^{2}\left( M\right) $ by $u\rightarrow
P_{g}(u)=\Delta ^{2}u-\nabla ^{i}(a\rho ^{-\mu }\nabla _{i}u)+\rho ^{-\alpha
}bu$. Our main results state as follows:

\begin{theorem}
\label{Th1} Let $0<\sigma <2$ and $0<\mu <4$. Suppose that the operator $%
P_{g}$ is coercive and 
\begin{equation}
\left\{ 
\begin{array}{c}
\frac{\Delta f(x_{o})}{f\left( x_{o}\right) }<\frac{1}{3}\left( \frac{%
(n-1)n\left( n^{2}+4n-20\right) }{\left( n^{2}-4\right) \left( n-4\right)
\left( n-6\right) }\frac{1}{\left( 1+\left\Vert \frac{a}{\rho ^{\sigma }}%
\right\Vert _{r}+\left\Vert \frac{b}{\rho ^{\mu }}\right\Vert _{s}\right) ^{%
\frac{4}{n}}}-1\right) S_{g}\left( x_{o}\right) \text{ in case }n>6 \\ 
S_{g}(x_{\circ })>0\text{ \ in case }n=6\text{.}%
\end{array}%
\right.  \tag{C}  \label{C}
\end{equation}%
Then there is $\lambda _{\ast }>0$ such that if $\lambda \in (0,$ $\lambda
_{\ast })$, the equation (\ref{3'}) possesses at least two distinct non
trivial solutions in the \ distribution sense.
\end{theorem}

The proof of Theorem \ref{Th1} relies on the following Hardy-Sobolev
inequality ( see \cite{4}).

\begin{lemma}
\label{Lem} Let $\left( M,g\right) $ be a compact $n$- dimensional
Riemannian manifold, and $p$, $q$ and $\gamma $ real numbers satisfying

\begin{equation*}
1\leq q\leq p\leq \frac{nq}{n-2q}\text{, }n>2q\text{, }\frac{\gamma }{p}%
=-2+n\left( \frac{1}{q}-\frac{1}{p}\right) >-\frac{n}{p}\text{.}
\end{equation*}%
For any $\varepsilon >0$, there is a constant $A(\varepsilon ,q,\gamma )$
such that

\begin{equation*}
\forall f\in H_{2}^{q}(M)\text{, \ }\left\Vert f\right\Vert _{p,\rho
^{\gamma }}^{q}\leq \left( 1+\varepsilon \right) K^{q}(n,q,\gamma
)\left\Vert \nabla ^{2}f\right\Vert _{q}^{q}+A(\varepsilon ,q,\gamma
)\left\Vert f\right\Vert _{q}^{q}\text{.}
\end{equation*}%
In particular in case $\gamma =0$, $K(n,q,0)=K(n,q)$ is the best constant in
Sobolev's inequality.
\end{lemma}

For brevity along all this work we put $K_{o}=K(n,2)$.

Let $\sigma $ and $\mu $ be as in Theorem \ref{Th1}, the Hardy- Sobolev
inequality given by Lemma \ref{Lem} leads to%
\begin{equation*}
\int_{M}\frac{\left\vert \nabla u\right\vert ^{2}}{\rho ^{\sigma }}%
dv_{g}\leq C(\left\Vert \nabla \left\vert \nabla u\right\vert \right\Vert
^{2}+\left\Vert \nabla u\right\Vert ^{2})
\end{equation*}%
and since 
\begin{equation*}
\left\Vert \nabla \left\vert \nabla u\right\vert \right\Vert ^{2}\leq
\left\Vert \nabla ^{2}u\right\Vert ^{2}\leq \left\Vert \Delta u\right\Vert
^{2}+\beta \left\Vert \nabla u\right\Vert ^{2}
\end{equation*}%
where $\beta >0$ is a constant and it is well known that for any $%
\varepsilon >0$ there is a constant $c\left( \varepsilon \right) >0$ such
that 
\begin{equation*}
\left\Vert \nabla u\right\Vert ^{2}\leq \varepsilon \left\Vert \Delta
u\right\Vert ^{2}+c\left\Vert u\right\Vert ^{2}\text{.}
\end{equation*}%
Hence 
\begin{equation}
\int_{M}\frac{\left\vert \nabla u\right\vert ^{2}}{\rho ^{\sigma }}%
dv_{g}\leq C\left( 1+\varepsilon \right) \left\Vert \Delta u\right\Vert
^{2}+A\left( \varepsilon ,\sigma \right) \left\Vert u\right\Vert ^{2}\text{.}
\label{3iv}
\end{equation}%
Let $K(n,2,\sigma )$ be the best constant in inequality (\ref{3iv}) and $%
K(n,2,\mu )$ the best one in the inequality%
\begin{equation*}
\int_{M}\frac{\left\vert u\right\vert ^{2}}{\rho ^{\mu }}dv_{g}\leq C\left(
1+\varepsilon \right) \left\Vert \Delta u\right\Vert ^{2}+A\left(
\varepsilon ,\mu \right) \left\Vert u\right\Vert ^{2}\text{.}
\end{equation*}%
For any $0<\sigma <2$ and $0<\mu <4$, let $u_{\sigma ,\mu \text{ }}$be the
solution of Equation (\ref{3}). In the sharp case $\sigma =2$ and $\mu =4$,
we obtain the following result

\begin{theorem}
Let $(M,g)$ be a Riemannian compact manifold of dimension $n\geq 5$. Suppose
that the operator $P_{g}$ is coercive and let $\left( u_{_{\sigma ,\mu
}}\right) _{\sigma ,\mu }$ be a sequence in $M_{\lambda }$ such that%
\begin{equation*}
\left\{ 
\begin{array}{c}
J_{\lambda ,\sigma ,\mu }(u_{_{\sigma ,\mu _{{}}}})\leq c_{\sigma ,\mu } \\ 
\nabla J_{\lambda }(u_{_{\sigma ,\mu }})-\mu _{_{\sigma ,\mu }}\nabla \Phi
_{\lambda }(u_{_{\sigma ,\mu }})\rightarrow 0%
\end{array}%
\right. \text{.}
\end{equation*}%
Suppose that%
\begin{equation*}
c_{\sigma ,\mu }<\frac{2}{n\text{ }K_{o}^{\frac{n}{4}}(f(x_{\circ }))^{\frac{%
n-4}{4}}}
\end{equation*}%
and 
\begin{equation*}
1+a^{-}\max \left( K(n,2,\sigma ),A\left( \varepsilon ,\sigma \right)
\right) +b^{-}\max \left( K(n,2,\mu ),A\left( \varepsilon ,\mu \right)
\right) >0
\end{equation*}%
then the equation%
\begin{equation*}
\Delta ^{2}u-\nabla ^{\mu }(\frac{a}{\rho ^{2}}\nabla _{\mu }u)+\frac{bu}{%
\rho ^{4}}=f\left\vert u\right\vert ^{N-2}u+\lambda \left\vert u\right\vert
^{q-2}u
\end{equation*}%
has at least two distinct non trivial solutions in distribution sense.
\end{theorem}

We consider the energy functional $J_{\lambda }$ defined by for each $u\in
H_{2}^{2}(M)$ by%
\begin{equation*}
J_{\lambda }(u)=\frac{1}{2}\int_{M}\left( \left( \Delta _{g}u\right)
^{2}-a(x)\rho ^{-\sigma }\left\vert \nabla _{g}u\right\vert ^{2}+b(x)\rho
^{-\mu }u^{2}\right) dv(g)-\frac{\lambda }{q}\int_{M}\left\vert u\right\vert
^{q}dv(g)
\end{equation*}%
\begin{equation*}
-\frac{1}{N}\int_{M}f(x)\left\vert u\right\vert ^{N}dv(g)\text{.}
\end{equation*}%
Put 
\begin{equation*}
\Phi _{\lambda }(u)=\left\langle \nabla J_{\lambda }(u),\text{ }%
u\right\rangle
\end{equation*}%
\begin{equation*}
\Phi _{\lambda }(u)=\int_{M}\left( \left( \Delta _{g}u\right) ^{2}-a(x)\rho
^{-\sigma }\left\vert \nabla _{g}u\right\vert ^{2}+b(x)\rho ^{-\mu
}u^{2}\right) dv(g)-\lambda \int_{M}\left\vert u\right\vert ^{q}dv(g)
\end{equation*}%
\begin{equation*}
-\int_{M}f(x)\left\vert u\right\vert ^{N}dv(g)
\end{equation*}%
and 
\begin{equation*}
\left\langle \nabla \Phi _{\lambda }(u),u\right\rangle =2\int_{M}\left(
\left( \Delta _{g}u\right) ^{2}-a(x)\rho ^{-\sigma }\left\vert \nabla
_{g}u\right\vert ^{2}+b(x)\rho ^{-\mu }u^{2}\right) dv(g)-\lambda
q\int_{M}\left\vert u\right\vert ^{q}dv(g)
\end{equation*}%
\begin{equation*}
-\lambda q\int_{M}\left\vert u\right\vert ^{q}dv(g)-N\int_{M}f(x)\left\vert
u\right\vert ^{N}dv(g)\text{.}
\end{equation*}

It is well-known that the solutions of equation (\ref{3'}) are critical
points of the energy functional $J_{\lambda }$. The Nehari minimization
problem writes as follows%
\begin{equation*}
\alpha _{\lambda }=\inf \text{ }\left\{ J_{\lambda }(u):u\in N_{\lambda
}\right\} =\inf_{u\in N_{\lambda }}J_{\lambda }(u)
\end{equation*}%
where 
\begin{equation*}
N_{\lambda }=\left\{ u\in H_{2}^{2}(M)\backslash \left\{ 0\right\} :\text{ }%
\Phi _{\lambda }(u)=0\right\} \text{.}
\end{equation*}

Note that $N_{\lambda }$ contains every solution of\ equation (\ref{3'}).$%
\newline
$ $N_{\lambda }$ splits in three parts%
\begin{equation*}
N_{\lambda }^{+}=\left\{ u\in N_{\lambda }:\text{ }\left\langle \nabla \Phi
_{\lambda }(u),\text{ }u\right\rangle >0\right\}
\end{equation*}%
\begin{equation*}
N_{\lambda }^{-}=\left\{ u\in N_{\lambda }:\text{ }\left\langle \nabla \Phi
_{\lambda }(u),\text{ }u\right\rangle <0\right\}
\end{equation*}%
\begin{equation*}
N_{\lambda }^{0}=\left\{ u\in N_{\lambda }:\text{ }\left\langle \nabla \Phi
_{\lambda }(u),\text{ }u\right\rangle =0\right\} \text{.}
\end{equation*}%
Before stating our main result, we give some nice properties of $N_{\lambda
}^{+}$, $N_{\lambda }^{-}$ and $N_{\lambda }^{0}.\newline
$Let 
\begin{equation}
\lambda _{\circ }=\frac{\left( N-2\right) q\text{ }\Lambda ^{\frac{q}{2}}}{%
2\left( N-q\right) V(M)^{1-\frac{q}{N}}(\max (K_{\circ },A_{\varepsilon }))^{%
\frac{q}{2}}}  \label{3''}
\end{equation}%
The following lemma shows that the minimizers of $J_{\lambda }$ on $%
N_{\lambda }$ are usually critical points for $J_{\lambda }$.

\begin{lemma}
\label{lem1} Let $\lambda $ $\in (0,$ $\lambda _{\circ })$, if $v$ is a
local minimizer for $J_{\lambda }$on $N_{\lambda }$ and $v\notin N_{\lambda
}^{0}$, then $\nabla J_{\lambda }(v)=0$.
\end{lemma}

\begin{proof}
If\ $v$ is a local minimizer for $J_{\lambda }$on $N_{\lambda }$, then by
Lagrange multipliers' theorem, there is a real $\theta $ such that for any $%
\varphi \in H_{2}^{2}(M)$%
\begin{equation*}
\left\langle \nabla J_{\lambda }(v),\varphi \right\rangle =\theta
\left\langle \nabla \Phi _{\lambda }(v),\varphi \right\rangle
\end{equation*}%
If $\theta =0$, then the lemma is proved. If it is not the case we pick $%
\varphi =v$ and we use the assumption that $v\in N_{\lambda }$ to infer that%
\begin{equation*}
\left\langle \nabla J_{\lambda }(v),v\right\rangle =\theta \left\langle
\nabla \Phi _{\lambda }(v),v\right\rangle =0
\end{equation*}%
which contradicts that $v\notin N_{\lambda }^{0}$.
\end{proof}

Now we give some preparatory lemmas.

\begin{lemma}
\label{lem2} There is $\lambda _{1}>0$ such that for any $\lambda \in \left(
0,\lambda _{1}\right) $ the set $N_{\lambda }^{0}$ is empty .
\end{lemma}

\begin{proof}
Suppose for every $\lambda >0$ there is $\lambda ^{\prime }\in \left(
0,\lambda \right) $ such that $N_{\lambda ^{\prime }}^{0}$ $\neq \emptyset $
and let $u\in N_{\lambda ^{\prime }}^{0}$ i.e.%
\begin{equation*}
\left\langle \nabla \Phi _{\lambda ^{\prime }}(u),u\right\rangle
=2\left\Vert u\right\Vert ^{2}-\lambda ^{\prime }q\left\Vert u\right\Vert
_{q}^{q}-N\int_{M}f(x)\left\vert u\right\vert ^{N}dv(g)=0
\end{equation*}%
and by the fact that 
\begin{equation*}
\Phi _{\lambda ^{\prime }}(u)=\left\Vert u\right\Vert ^{2}-\lambda ^{\prime
}\left\Vert u\right\Vert _{q}^{q}-\int_{M}f(x)\left\vert u\right\vert
^{N}dv(g)=0
\end{equation*}%
we get 
\begin{equation}
\left\Vert u\right\Vert ^{2}=\frac{N-q}{2-q}\int_{M}f(x)\left\vert
u\right\vert ^{N}dv(g)  \label{2}
\end{equation}%
and also 
\begin{equation}
\lambda ^{\prime }\left\Vert u\right\Vert _{q}^{q}=\frac{N-2}{2-q}%
\int_{M}f(x)\left\vert u\right\vert ^{N}dv(g)\text{.}  \label{3}
\end{equation}%
Independently by the Sobolev's inequality and the coerciveness of the
operator $P_{g}$ we obtain%
\begin{equation}
\int_{M}f(x)\left\vert u\right\vert ^{N}dv(g)\leq \Lambda ^{-\frac{N}{2}%
}(\max ((1+\varepsilon )K_{\circ },A_{\varepsilon }))^{\frac{N}{2}%
}\max_{x\in M}f(x)\left\Vert u\right\Vert ^{N}  \label{4}
\end{equation}%
where $\Lambda $ denotes a constant of the coercivity. From (\ref{2}) and (%
\ref{4}) we deduce that 
\begin{equation*}
\left\Vert u\right\Vert \geq \left[ \frac{\left( N-q\right) \Lambda ^{-\frac{%
N}{2}}\left( (\max ((1+\varepsilon )K_{\circ },A_{\varepsilon })\right) ^{%
\frac{N}{2}}\max_{x\in M}f(x)}{\left( 2-q\right) }\right] ^{\frac{1}{2-N}}
\end{equation*}%
Let the functional $I_{\lambda ^{\prime }}:N_{\lambda }\rightarrow 
\mathbb{R}
$ is given by 
\begin{equation*}
I_{\lambda ^{\prime }}(u)=\left[ \left( \frac{N-q}{2-q}\right) ^{\frac{q}{2}}%
\frac{2-q}{N-2}\right] ^{\frac{2}{2-q}}\left( \frac{\left\Vert u\right\Vert
^{q}}{\lambda ^{\prime }\left\Vert u\right\Vert _{q}^{q}}\right) ^{\frac{2}{%
q-2}}-\int_{M}f(x)\left\vert u\right\vert ^{N}dv(g).
\end{equation*}%
If $u\in N_{\lambda ^{\prime }}^{0}$, then (\ref{2}) and (\ref{3}) give \ \ 
\begin{equation*}
I_{\lambda ^{\prime }}(u)=\left[ \left( \frac{N-q}{2-q}\right) ^{\frac{q}{2}}%
\frac{2-q}{N-2}\right] ^{\frac{2}{2-q}}\left[ \frac{\left( \frac{N-q}{2-q}%
\int_{M}f(x)\left\vert u\right\vert ^{N}dv(g)\right) ^{\frac{q}{2}}}{\frac{%
N-2}{2-q}\int_{M}f(x)\left\vert u\right\vert ^{N}dv(g)}\right] ^{\frac{2}{q-2%
}}
\end{equation*}%
\begin{equation}
-\int_{M}f(x)\left\vert u\right\vert ^{N}dv(g)=0\text{.}  \label{6}
\end{equation}%
Putting 
\begin{equation*}
\theta =\left[ \left( \frac{N-q}{2-q}\right) ^{\frac{q}{2}}\frac{2-q}{N-2}%
\right] ^{\frac{2}{2-q}}
\end{equation*}%
and taking account of the coerciveness of the operator $P_{g}$ and the
Sobolev's inequality one get%
\begin{equation*}
I_{\lambda ^{\prime }}(u)\geq \theta \left( \frac{\left\Vert u\right\Vert
^{q}}{\lambda \frac{N-q}{Nq}\Lambda ^{-\frac{q}{2}}V(M)^{1-\frac{q}{N}}(\max
((1+\varepsilon )K_{\circ },A_{\varepsilon }))^{\frac{q}{2}}\left\Vert
u\right\Vert ^{q}}\right) ^{\frac{2}{q-2}}
\end{equation*}%
\begin{equation*}
-\Lambda ^{-\frac{N}{2}}(\max ((1+\varepsilon )K_{\circ },A_{\varepsilon
}))^{\frac{N}{2}}\max_{x\in M}f(x)\left\Vert u\right\Vert ^{N}\text{.}
\end{equation*}%
That is to say 
\begin{equation*}
I_{\lambda ^{\prime }}(u)\geq \left( \frac{\Lambda ^{\frac{q}{2}}\left( 
\frac{N-q}{2-q}\right) ^{\frac{q}{2}}\left( \frac{2-q}{N-2}\right) \left( 
\frac{Nq}{N-q}\right) }{\lambda ^{\prime }V(M)^{1-\frac{q}{N}}(\max
((1+\varepsilon )K_{\circ },A_{\varepsilon }))^{\frac{q}{2}}}\right) ^{\frac{%
2}{q-2}}
\end{equation*}%
\begin{equation*}
-\left( \left( \frac{N-q}{2-q}\right) \Lambda ^{-\frac{N}{2}}\left( (\max
((1+\varepsilon )K_{\circ },A_{\varepsilon })\right) ^{\frac{N}{2}%
}\max_{x\in M}f(x)\right) ^{\frac{2}{2-N}}\text{.}
\end{equation*}%
Hence, if $\lambda $ is sufficiently small, so as $\lambda ^{\prime }>0$ and 
$I_{\lambda ^{\prime }}(u)>0$ for all $u\in N_{\lambda ^{\prime }}^{0}$.
This contradicts (\ref{6}). So there is $\lambda _{1}>0$, such that for any $%
\lambda \in (0,$ $\lambda _{1})$, \ the set $N_{\lambda }^{0}=\emptyset .$
\end{proof}

From Lemma \ref{lem2}, $N_{\lambda }$ splits as $N_{\lambda }=N_{\lambda
}^{+}\cup N_{\lambda }^{-}$ where $0<\lambda <\lambda _{1}$. We define 
\begin{equation*}
\alpha _{\lambda }=\inf_{u\in N_{\lambda }}J_{\lambda }(u)\text{, \ }\alpha
_{\lambda }^{+}=\inf_{u\in N_{\lambda }^{+}}J_{\lambda }(u)\text{ \ and \ }%
\alpha _{\lambda }^{-}=\inf_{u\in N_{\lambda }^{-}}J_{\lambda }(u)
\end{equation*}

\begin{lemma}
\label{lem3} For each $\lambda $ $\in (0,$ $\lambda _{\circ }),$ the
functional $J_{\lambda }$ is bounded from below on $N_{\lambda }$.
\end{lemma}

\begin{proof}
If $u\in N_{\lambda },$then by equality (\ref{2}) and the Sobolev's
inequality, we deduce that 
\begin{equation*}
J_{\lambda }(u)\geq \frac{N-2}{2N}\left\Vert u\right\Vert ^{2}-\lambda \frac{%
N-q}{Nq}V(M)^{1-\frac{q}{N}}(\max ((1+\varepsilon )K_{\circ },A_{\varepsilon
}))^{\frac{q}{2}}\left\Vert u\right\Vert _{H_{2}^{2}(M)}^{q}
\end{equation*}%
and taking account of the coerciveness of the operator $P_{g}$, we infer
that 
\begin{equation*}
J_{\lambda }(u)\geq \frac{N-2}{2N}\left\Vert u\right\Vert ^{2}-\lambda \frac{%
N-q}{Nq}\Lambda ^{-\frac{q}{2}}V(M)^{1-\frac{q}{N}}(\max ((1+\varepsilon
)K_{\circ },A_{\varepsilon }))^{\frac{q}{2}}\left\Vert u\right\Vert ^{q}
\end{equation*}%
where $\Lambda $ is a constant of coercivity.

If $u\in N_{\lambda }$ and $\left\Vert u\right\Vert \geq 1$, 
\begin{equation*}
J_{\lambda }(u)\geq \left[ \frac{N-2}{2N}-\lambda \frac{N-q}{Nq}\Lambda ^{-%
\frac{q}{2}}V(M)^{1-\frac{q}{N}}(\max ((1+\varepsilon )K_{\circ
},A_{\varepsilon }))^{\frac{q}{2}}\right] \left\Vert u\right\Vert ^{q}
\end{equation*}%
So, if 
\begin{equation*}
0<\lambda <\frac{\left( N-2\right) q\text{ }\Lambda ^{\frac{q}{2}}}{2\left(
N-q\right) V(M)^{1-\frac{q}{N}}(\max (K_{\circ },A_{\varepsilon }))^{\frac{q%
}{2}}}:=\lambda _{\circ }
\end{equation*}%
then%
\begin{equation*}
J_{\lambda }(u)>0
\end{equation*}%
If $u\in N_{\lambda }$ with$\left\Vert u\right\Vert <1$, we have%
\begin{equation*}
J_{\lambda }(u)>-\lambda \frac{N-q}{Nq}\Lambda ^{-\frac{q}{2}}V(M)^{1-\frac{q%
}{N}}(\max ((1+\varepsilon )K_{\circ },A_{\varepsilon }))^{\frac{q}{2}}\text{%
.}
\end{equation*}%
Thus $J_{\lambda }$ is bounded below on $N_{\lambda }$ .
\end{proof}

As a consequence of Lemma \ref{lem1} we have

\begin{lemma}
\label{lem4} If $\ \lambda \in \left( 0,\lambda _{\circ }\right) $, we have
\end{lemma}

\begin{equation*}
\alpha _{\lambda }^{+}=\inf_{u\in N_{\lambda }^{+}}J_{\lambda }(u)<0.
\end{equation*}

\begin{proof}
If $u\in N_{\lambda }^{+}$, then 
\begin{equation*}
J_{\lambda }(u)=\frac{N-2}{2N}\left\Vert u\right\Vert ^{2}-\frac{\lambda
(N-q)}{Nq}\left\Vert u\right\Vert _{q}^{q}
\end{equation*}%
and since%
\begin{equation*}
\left\langle \nabla \Phi _{\lambda }(u),u\right\rangle =2\left\Vert
u\right\Vert ^{2}-\lambda q\left\Vert u\right\Vert
_{q}^{q}-N\int_{M}f(x)\left\vert u\right\vert ^{N}dv(g)>0
\end{equation*}%
we get%
\begin{equation*}
J_{\lambda }(u)\leq \frac{\lambda (N-q)}{N}\left( \frac{1}{2}-\frac{1}{q}%
\right) \left\Vert u\right\Vert _{q}^{q}<0
\end{equation*}%
i.e.%
\begin{equation*}
\inf_{u\in N_{\lambda }^{+}}J_{\lambda }(u)<0\text{.}
\end{equation*}
\end{proof}

\begin{lemma}
\label{lem5} For every $\lambda \in \left( 0,\min (\lambda _{0},\lambda
_{1})\right) $, 
\begin{equation*}
\alpha _{\lambda }^{-}=\inf_{u\in N_{\lambda }^{-}}J_{\lambda }(u)>0\text{.}
\end{equation*}
\end{lemma}

\begin{proof}
If $u\in N_{\lambda }^{-}$, then 
\begin{equation*}
J_{\lambda }(u)=\frac{N-2}{2N}\left\Vert u\right\Vert ^{2}-\frac{\lambda
(N-q)}{Nq}\left\Vert u\right\Vert _{q}^{q}
\end{equation*}%
and since%
\begin{equation}
\left\langle \nabla \Phi _{\lambda }(u),u\right\rangle =2\left\Vert
u\right\Vert ^{2}-\lambda q\left\Vert u\right\Vert
_{q}^{q}-N\int_{M}f(x)\left\vert u\right\vert ^{N}dv(g)<0  \label{7}
\end{equation}%
we infer that%
\begin{equation}
\left\Vert u\right\Vert ^{2}>\frac{\lambda \left( N-q\right) }{\left(
N-2\right) }\left\Vert u\right\Vert _{q}^{q}\text{.}  \label{8}
\end{equation}%
By Sobolev's inequality and from the coerciveness of the operator $P_{g}$,
there exists a constant $\Lambda >0$, such that%
\begin{equation*}
J_{\lambda }(u)\geq \frac{N-2}{2N}\left\Vert u\right\Vert ^{2}-\lambda \frac{%
N-q}{Nq}\Lambda ^{-\frac{q}{2}}V(M)^{1-\frac{q}{N}}(\max ((1+\varepsilon
)K_{\circ },A_{\varepsilon }))^{\frac{q}{2}}\left\Vert u\right\Vert ^{q}%
\text{.}
\end{equation*}%
So if $u\in N_{\lambda }^{-}$ and $\left\Vert u\right\Vert \geq 1$, 
\begin{equation}
J_{\lambda }(u)\geq \left[ \frac{N-2}{2N}-\lambda \frac{N-q}{Nq}\Lambda ^{-%
\frac{q}{2}}V(M)^{1-\frac{q}{N}}(\max ((1+\varepsilon )K_{\circ
},A_{\varepsilon }))^{\frac{q}{2}}\right] \left\Vert u\right\Vert ^{q}
\label{8'}
\end{equation}%
hence if 
\begin{equation*}
0<\lambda <\frac{\left( N-2\right) q\text{ }\Lambda ^{\frac{q}{2}}}{2\left(
N-q\right) V(M)^{1-\frac{q}{N}}(\max ((1+\varepsilon )K_{\circ
},A_{\varepsilon }))^{\frac{q}{2}}}=\lambda _{\circ }
\end{equation*}%
then 
\begin{equation*}
J_{\lambda }(u)>0
\end{equation*}%
In case $u\in N_{\lambda }^{-}$ and $\left\Vert u\right\Vert <1$, by
Sobolev's inequality, the inequality (\ref{7}) and the coerciveness of the
operator $P_{g}$, we obtain%
\begin{equation*}
0<\xi \leq \left\Vert u\right\Vert <1
\end{equation*}%
where%
\begin{equation*}
\xi =\left[ \frac{\left( 2-q\right) \Lambda ^{\frac{N}{2}}(\max
((1+\varepsilon )K_{\circ },A_{\varepsilon }))^{-\frac{N}{2}}}{\left(
N-q\right) \max_{x\in M}f(x)}\right] ^{\frac{1}{N-2}}
\end{equation*}%
and $\Lambda $ is a constant of coerciveness.

The inequality (\ref{8'}) becomes%
\begin{equation*}
J_{\lambda }(u)\geq \frac{N-2}{2N}\xi ^{2}-\lambda \frac{N-q}{Nq}\Lambda ^{-%
\frac{q}{2}}V(M)^{1-\frac{q}{N}}(\max ((1+\varepsilon )K_{\circ
},A_{\varepsilon }))^{\frac{q}{2}}
\end{equation*}%
Hence, if we take 
\begin{equation}
0<\lambda <\frac{\frac{\left( N-2\right) }{2\left( N-q\right) }\xi
^{2}\Lambda ^{\frac{q}{2}}}{V(M)^{1-\frac{q}{N}}(\max (K_{\circ
},A_{\varepsilon }))^{\frac{q}{2}}}=\lambda _{2}  \label{8"}
\end{equation}%
then if $\ \lambda \in \left( 0,\min (\lambda _{0},\lambda _{1},\lambda
_{2})\right) $ we obtain 
\begin{equation*}
J_{\lambda }(u)\geq C>0
\end{equation*}%
where $C$ is constant depending on $N$, $\Lambda $, $V(M)$, $K_{o}$ and $%
A_{\varepsilon }$. So 
\begin{equation*}
\inf_{u\in N_{\lambda }^{-}}J_{\lambda }(u)>0\text{.}
\end{equation*}
\end{proof}

For each $u\in H_{2}-\left\{ 0\right\} $, define%
\begin{equation*}
E(t)=t^{2-q}\left\Vert u\right\Vert ^{2}-t^{N-q}\int_{M}f\left\vert
u\right\vert ^{N}dv_{g}
\end{equation*}%
so $E(0)=0$ and $E(t)$ goes to $-\infty $ as $t\rightarrow +\infty $. Also
for $t>0$, we have 
\begin{equation*}
E^{\prime }(t)=\left( 2-q\right) t^{1-q}\left\Vert u\right\Vert
^{2}-(N-q)t^{N-q-1}\int_{M}f\left\vert u\right\vert ^{N}dv_{g}
\end{equation*}%
and $E^{\prime }(t)=0$ at 
\begin{equation*}
t_{o}=\left( \frac{2-q}{N-q}\right) ^{\frac{1}{N-2}}\left( \frac{\left\Vert
u\right\Vert ^{2}}{\int_{M}f\left\vert u\right\vert ^{N}dv_{g}}\right) ^{%
\frac{1}{N-2}}\text{.}
\end{equation*}%
Hence $E(t)$ achieves its maximum at $t_{o\text{ }}$and it is increasing on $%
\left[ 0,t_{o}\right) $ and decreasing on $\left[ t_{o}\text{, }+\infty
\right) $.

Evaluating the function $E$ at $t_{o}$, 
\begin{equation*}
E(t_{o})=\left( \frac{2-q}{N-q}\right) ^{\frac{2-q}{N-2}}\left( \frac{%
\left\Vert u\right\Vert ^{2}}{\int_{M}f\left\vert u\right\vert ^{N}dv_{g}}%
\right) ^{\frac{2-q}{N-q}}\left\Vert u\right\Vert ^{2}
\end{equation*}%
\begin{equation*}
-\left( \frac{2-q}{N-q}\right) ^{\frac{N-q}{N-2}}\left( \frac{\left\Vert
u\right\Vert ^{2}}{\int_{M}f\left\vert u\right\vert ^{N}dv_{g}}\right) ^{^{%
\frac{N-q}{N-2}}}\int_{M}f\left\vert u\right\vert ^{N}dv_{g}
\end{equation*}%
\begin{equation*}
=\frac{N-2}{N-q}\left( \frac{2-q}{N-q}\right) ^{\frac{2-q}{N-2}}\frac{%
\left\Vert u\right\Vert ^{\frac{2(N-q)}{N-2}}}{\left( \int_{M}f\left\vert
u\right\vert ^{N}dv_{g}\right) ^{\frac{2-q}{N-2}}}\text{.}
\end{equation*}%
By the Sobolev's inequality we get for any $\epsilon >0,$%
\begin{equation*}
\int_{M}f\left\vert u\right\vert ^{N}dv_{g}\leq \left\Vert f\right\Vert
_{\infty }\left( \left( K_{o}^{2}+\epsilon \right) \left\Vert \Delta
u\right\Vert _{2}^{2}+A\left( \epsilon \right) \left\Vert u\right\Vert
_{2}^{2}\right) ^{\frac{N}{2}}
\end{equation*}%
\begin{equation*}
\leq \left\Vert f\right\Vert _{\infty }\max \left( K_{o}^{2}+\epsilon
,A\left( \epsilon \right) \right) ^{\frac{N}{2}}\left\Vert u\right\Vert
_{H_{2}^{2}}^{N}
\end{equation*}%
\begin{equation*}
\leq \Lambda ^{-\frac{N}{2}}\left\Vert f\right\Vert _{\infty }\max \left(
K_{o}^{2}+\epsilon ,A\left( \epsilon \right) \right) ^{\frac{N}{2}%
}\left\Vert u\right\Vert ^{N}
\end{equation*}%
\begin{equation*}
=C^{\frac{N}{2}}\left\Vert f\right\Vert _{\infty }\left\Vert u\right\Vert
^{N}
\end{equation*}%
where $\Lambda $ is the constant of the coercivity, $K_{o}$ the best
constant in the Sobolev's inequality and $A\left( \epsilon \right) $ the
correspondent constant, $\left\Vert f\right\Vert _{\infty }=\sup_{x\in
M}\left\vert f(x)\right\vert $ and $C=\Lambda ^{-1}\max \left(
K_{o}^{2}+\epsilon ,A\left( \epsilon \right) \right) $.

Consequently%
\begin{equation}
E\left( t_{o}\right) \geq \frac{N-2}{N-q}\left( \frac{2-q}{N-q}\right) ^{%
\frac{2-q}{N-q}}C^{\frac{N(q-2)}{2\left( N-2\right) }}\left\Vert
f\right\Vert _{\infty }\left\Vert u\right\Vert ^{q}\text{.}  \label{8''}
\end{equation}%
Independently and in the same way as above we get%
\begin{equation}
\left\Vert u\right\Vert _{q}^{q}\leq \Lambda ^{-\frac{q}{2}}vol(M)^{1-\frac{q%
}{N}}C^{\frac{q}{2}}\left\Vert u\right\Vert ^{q}\text{.}  \label{8'''}
\end{equation}%
Hence 
\begin{equation*}
E(0)=0<\lambda \left\Vert u\right\Vert _{q}^{q}\leq E(t_{o})
\end{equation*}%
provided that%
\begin{equation*}
\lambda \leq \frac{\frac{N-2}{N-q}\left( \frac{2-q}{N-q}\right) ^{\frac{2-q}{%
N-q}}\left\Vert f\right\Vert _{\infty }}{vol(M)^{1-\frac{q}{2}}C^{\frac{N-q}{%
N-2}}}\text{.}
\end{equation*}%
Consequently by the nature of the function $E(t)$ we infer the existence of $%
t^{-}$, $t^{+}$ with $0<t^{+}<t^{o}<t^{-}$ such that 
\begin{equation}
\lambda \left\Vert u\right\Vert _{q}^{q}=E(t^{+})=E(t^{-})\text{.}
\label{8''''}
\end{equation}%
and%
\begin{equation*}
E^{\prime }(t^{+})>0>E^{\prime }(t^{-})
\end{equation*}%
Now we evaluate $\Phi _{\lambda }$ at $t^{-}u$ and at $t^{+}u$ to get%
\begin{equation*}
\Phi _{\lambda }(t^{-}u)=\left\langle \nabla J_{\lambda }\left(
t^{-}u\right) ,t^{-}u\right\rangle
\end{equation*}%
\begin{equation*}
=\left( t^{-}\right) ^{2}\left\Vert u\right\Vert ^{2}-\left( t^{-}\right)
^{N}\int_{M}f\left\vert u\right\vert ^{N}dv_{g}-\lambda \left( t^{-}\right)
^{q}\left\Vert u\right\Vert _{q}^{q}
\end{equation*}%
\begin{equation*}
=\left( t^{-}\right) ^{q}\left( \left( t^{-}\right) ^{2-q}\left\Vert
u\right\Vert ^{2}-\left( t^{-}\right) ^{N-q}\int_{M}f\left\vert u\right\vert
^{N}dv_{g}-\lambda \left\Vert u\right\Vert _{q}^{q}\right)
\end{equation*}%
and by (\ref{8''''}) we deduce that 
\begin{equation*}
\Phi _{\lambda }(t^{-}u)=0
\end{equation*}%
and also we get%
\begin{equation*}
\Phi _{\lambda }(t^{+}u)=0\text{.}
\end{equation*}%
Moreover, we have%
\begin{equation*}
\left\langle \nabla \Phi _{\lambda }(t^{-}u),t^{-}u\right\rangle =2\left(
t^{-}\right) ^{2}\left\Vert u\right\Vert ^{2}-N\left( t^{-}\right)
^{N}\int_{M}f\left\vert u\right\vert ^{N}dv_{g}-q\left( t^{-}\right)
^{q}\lambda \left\Vert u\right\Vert _{q}^{q}
\end{equation*}%
and taking account of (\ref{8''''}) we infer that%
\begin{equation*}
\left\langle \nabla \Phi _{\lambda }(t^{-}u),t^{-}u\right\rangle =\left(
2-q\right) \left( t^{-}\right) ^{2}\left\Vert u\right\Vert ^{2}-\left(
N-q\right) \left( t^{-}\right) ^{N}\int_{M}f\left\vert u\right\vert
^{N}dv_{g}
\end{equation*}%
and again by (\ref{8''''}) we obtain%
\begin{equation*}
\left\langle \nabla \Phi _{\lambda }(t^{-}u),t^{-}u\right\rangle =\left(
t^{-}\right) ^{1+q}\left( \left( 2-q\right) \left( t^{-}\right)
^{1-q}\left\Vert u\right\Vert ^{2}-\left( N-q\right) \left( t^{-}\right)
^{N-q-1}\int_{M}f\left\vert u\right\vert ^{N}dv_{g}\right)
\end{equation*}%
\begin{equation*}
=\left( t^{-}\right) ^{1+q}E^{\prime }\left( t^{-}\right) <0
\end{equation*}%
that means that $t^{-}u\in N_{\lambda }^{-}$. By similar procedure we get
also $t^{+}u\in N_{\lambda }^{+}$.

\section{Existence of a local minimizer for $J_{\protect\lambda }$ on $N_{%
\protect\lambda }^{+}$ and $N_{\protect\lambda }^{-}$}

In this section we focus on the existence of a local minimum of $J_{\lambda
} $ on $N_{\lambda }^{+}$ and $N_{\lambda }^{-}$ to do so we will be in need
of the following Hardy-Sobolev inequality and Releich-Kondrakov embedding
respectively whose proofs are given in (\cite{6}). The weighted $L^{p}\left(
M,\rho ^{\gamma }\right) $ space will be the set of measurable functions $u$
on $M$ such that $\rho ^{\gamma }\left\vert u\right\vert ^{p}$ are
integrable where $p\geq 1$ and $\gamma $ are real numbers. We endow $%
L^{p}\left( M,\rho ^{\gamma }\right) $ with the norm

\begin{equation*}
\left\Vert u\right\Vert _{p,\rho }=\left( \int_{M}\rho ^{\gamma }\left\vert
u\right\vert ^{p}dv_{g}\right) ^{\frac{1}{p}}\text{.}
\end{equation*}

\begin{theorem}
\label{th1} Let $(M,g)$ be a Riemannian compact manifold of dimension $n\geq
5$ and $p$, $q$ , $\gamma $ are real numbers such that $\frac{\gamma }{p}=%
\frac{n}{q}-\frac{n}{p}-2$ $\ $and $2\leq p\leq \frac{2n}{n-4}$.$\newline
$For any $\ \epsilon >0$, there is $A(\epsilon ,q,\gamma )$ such that for
any $u\in H_{2}^{2}(M)$ 
\begin{equation*}
\left\Vert u\right\Vert _{p,\rho ^{\gamma }}^{2}\leq (1+\epsilon
)K(n,2,\gamma )^{2}\left\Vert \Delta _{g}u\right\Vert _{2}^{2}+A(\epsilon
,q,\gamma )\left\Vert u\right\Vert _{2}^{2}
\end{equation*}%
where $K(n,2,\gamma )$ is the optimal constant.
\end{theorem}

In case $\gamma =0$, $K(n,2,0)=K(n,2)=K_{o}^{\frac{1}{2}}$ is the best
constant in the Sobolev's embedding of $H_{2}^{2}(M)$ in $L^{N}(M)$ where $N=%
\frac{2n}{n-4}$.

\begin{theorem}
\label{th2} Let $(M,g)$ be a compact Riemannian manifold of dimension $n\geq
5$ and $p$, $q$, $\gamma $ are real numbers satisfying $1\leq q\leq p\leq 
\frac{nq}{n-2q}$ , $\gamma <0$ and $l=1$,$2$.

If $\frac{\gamma }{p}=n$ $(\frac{1}{q}-\frac{1}{p})-l$ then the inclusion $%
H_{l}^{q}(M)\subset $ $L^{p}(M,\rho ^{\gamma })$ is continuous.$\newline
$If $\frac{\gamma }{p}>n$ $(\frac{1}{q}-\frac{1}{p})-l$ then inclusion $%
H_{l}^{q}(M)\subset $ $L^{p}(M,\rho ^{\gamma })\ $is compact.
\end{theorem}

The following variant of the Ekeland's variational principle will be also
useful

\begin{lemma}
\label{lem5'} If $V$ is a Banach space and $J\in C^{1}\left( V,R\right) $ is
bounded from below, then there exists a minimizing sequence $\left(
u_{n}\right) $ for $J$ in $V$ such that \ $J(u_{n})\rightarrow \inf_{V}J$
and $E^{\prime }\left( u_{n}\right) \rightarrow 0$ as $n\rightarrow \infty $.
\end{lemma}

\begin{lemma}
\label{lem6} For any $\lambda $ $\in (0,$ $\lambda _{\circ })$

(i) \ There exists a minimizing sequence $\left( u_{m}\right) _{m}\subset
N_{\lambda }$ such that $J_{\lambda }(u_{m})=\alpha _{\lambda }+o(1)$ and $\
\nabla J_{\lambda }(u_{m})=o(1)$

(ii) There exists a minimizing sequence $\left( u_{m}\right) _{m}\subset
N_{\lambda }^{-}$ such that $J_{\lambda }(u_{m})=\alpha _{\lambda }^{-}+o(1)$
and $\nabla J_{\lambda }(u_{m})=o(1).$
\end{lemma}

\begin{proof}
By Lemma \ref{lem3} and the Enkland's variational principle ( see \ref{lem5'}%
) $J_{\lambda }$ admits a Palais-Smale sequence at level $\alpha _{\lambda }$
in $N_{\lambda }$( the same is also true for (ii) ).
\end{proof}

Now, we establish the existence of a local minimum for $J_{\lambda }$ on $%
N_{\lambda }^{+}$

\begin{theorem}
\label{thm2} Let $\lambda $ $\in (0,$ $\lambda _{\circ }),$ and suppose that
a sequence $(u_{m})_{m}$ $\subset N_{\lambda }^{+}$ fulfils

\begin{equation*}
\left\{ 
\begin{array}{c}
J_{\lambda }(u_{_{m}})\leq c \\ 
\nabla J_{\lambda }(u_{_{m}})-\mu _{_{m}}\nabla \Phi _{\lambda
}(u_{m})\rightarrow 0%
\end{array}%
\right.
\end{equation*}%
.

with%
\begin{equation}
c<\frac{2}{n\text{ }K_{\circ }^{\frac{n}{4}}(Max\text{ }_{x\in M}f(x))^{%
\frac{n-4}{4}}}.  \tag{C1}  \label{C1}
\end{equation}%
Then the functional $J_{\lambda }$\ has a minimizer $u^{+}$ in $N_{\lambda
}^{+}$ which satisfies

(i) $J_{\lambda }(u^{+})=\alpha _{\lambda }^{+}<0$,

(ii) $u^{+}$ is a nontrivial solution of equation (\ref{3'}).
\end{theorem}

\begin{proof}
Let $\left( u_{m}\right) _{m}\subset N_{\lambda }^{+}$ be a Palais-Smale
sequence for $J_{\lambda }$ on $N_{\lambda }$ i.e. 
\begin{equation*}
J_{\lambda }(u_{m})=\alpha _{\lambda }+o(1)\ \text{and\ }\nabla J_{\lambda
}(u_{m})=o(1)\text{ in }H_{2}^{2}(M)^{\prime }\text{.}
\end{equation*}%
Obviously 
\begin{equation*}
-\alpha _{\lambda }+o(1)\leq J_{\lambda }(u_{m})-\frac{1}{q}\left\langle
\nabla J_{\lambda }(u_{m}),u_{m}\right\rangle \leq \alpha _{\lambda }+o(1)
\end{equation*}%
or%
\begin{equation*}
-\alpha _{\lambda }+o(1)\leq \left( \frac{N-2}{2N}-\frac{N-2}{Nq}\right)
\left\Vert u_{m}\right\Vert ^{2}\leq \alpha _{\lambda }+o(1)\text{.}
\end{equation*}%
Hence%
\begin{equation*}
\alpha _{\lambda }\left( \frac{N-2}{Nq}-\frac{N-2}{2N}\right) ^{-1}+o(1)\leq
\left\Vert u_{m}\right\Vert ^{2}\leq -\alpha _{\lambda }\left( \frac{N-2}{Nq}%
-\frac{N-2}{2N}\right) ^{-1}+o(1)
\end{equation*}%
\newline
so the sequence $\left( u_{m}\right) _{m}$ is bounded in $H_{2}^{2}(M)$ and
by the well known Sobolev's embedding, we get up to a subsequence that

$\ \ \ \ \ \ \ $ $u_{m}\rightarrow u^{+}$\ \ weakly in $H_{2}^{2}(M)$.%
\newline
$\ \ \ \ \ \ \ $ $u_{m}\rightarrow u^{+}$ \ strongly in $L^{p}(M)$\ for $%
1<p<N=\frac{2n}{n-4}.$\newline
$\ \ \ \ \ \ $ $\nabla u_{m}\rightarrow \nabla u^{+}$ \ strongly in $%
L^{q}(M) $\ for $1<q<2^{\ast }=\frac{2n}{n-2}.$\newline
$\ \ \ \ \ \ \ $ $u_{m}\rightarrow u^{+}$ a.e in $M.$\newline
Put 
\begin{equation*}
w_{m}:=u_{m}-u^{+}
\end{equation*}%
by Br\'{e}zis-Lieb Lemma ( see \cite{9}), we obtain%
\begin{equation*}
\left\Vert \Delta _{g}u_{m}\right\Vert _{2}^{2}-\left\Vert \Delta
_{g}u^{+}\right\Vert _{2}^{2}=\left\Vert \Delta _{g}w_{m}\right\Vert
_{2}^{2}+o(1)
\end{equation*}%
and 
\begin{equation*}
\int_{M}f(x)\left( \left\vert u_{m}\right\vert ^{N}-\left\vert
u^{+}\right\vert ^{N}\right) dv(g)=\int_{M}f(x)\left\vert w_{m}\right\vert
^{N}dv(g)+o(1)
\end{equation*}%
Now since $\sigma \in \left( 0,2\right) $ and $\mu \in \left( 0,4\right) $,
by Theorem \ref{th2} we infer that $\nabla u_{m}\rightarrow \nabla u^{+}$
strongly in $L^{2}(M,\rho ^{-\sigma })$ and $u_{m}\rightarrow u^{+}$
strongly in $L^{2}(M,\rho ^{-\mu })$.First, we prove that $u^{+}\in
N_{\lambda }.$\newline
Taking into account of the strong convergences of $\nabla u_{m}\rightarrow
\nabla u^{+}$ in $L^{2}\left( M,\rho ^{-\sigma }\right) $ and $%
u_{m}\rightarrow u^{+}$ in $L^{2}\left( M,\rho ^{-\mu }\right) $, we obtain 
\begin{equation*}
J_{\lambda }(u_{m})-J_{\lambda }(u^{+})
\end{equation*}%
\begin{equation}
=\frac{1}{2}\left\Vert \Delta _{g}\left( u_{m}-u^{+}\right) \right\Vert
_{2}^{2}-\frac{1}{N}\int_{M}f(x)\left\vert u_{m}-u^{+}\right\vert
^{N}dv(g)+o(1)\text{.}  \label{9}
\end{equation}%
Since $u_{m}-u^{+}\rightarrow 0$\ weakly in $H_{2}^{2}(M)$, we test by $%
\nabla J_{\lambda }(u_{m})-\nabla J_{\lambda }(u)$ and get%
\begin{equation*}
\left\langle \nabla J_{\lambda }(u_{m})-\nabla J_{\lambda
}(u^{+}),u_{m}-u^{+}\right\rangle =
\end{equation*}%
\begin{equation}
\left\Vert \Delta _{g}\left( u_{m}-u^{+}\right) \right\Vert
_{2}^{2}-\int_{M}f(x)\left\vert u_{m}-u^{+}\right\vert ^{N}dv(g)=o(1)\text{.}
\label{10}
\end{equation}%
So by (\ref{10}), we obtain%
\begin{equation*}
J_{\lambda }(u_{m})-J_{\lambda }(u^{+})=\frac{1}{2}\left\Vert \Delta
_{g}\left( u_{m}-u^{+}\right) \right\Vert _{2}^{2}-\frac{1}{N}\left\Vert
\Delta _{g}\left( u_{m}-u^{+}\right) \right\Vert _{2}^{2}+o(1)
\end{equation*}%
i.e.%
\begin{equation*}
J_{\lambda }(u_{m})-J_{\lambda }(u^{+})=\frac{2}{n}\left\Vert \Delta
_{g}\left( u_{m}-u^{+}\right) \right\Vert _{2}^{2}+o(1)\text{.}
\end{equation*}%
By Sobolev's inequality, we have for all $u\in H_{2}^{2}(M)$ 
\begin{equation*}
\left\Vert u\right\Vert _{N}^{2}\leq (1+\varepsilon )K_{\circ
}\int_{M}\left( \Delta _{g}u\right) ^{2}+\left\vert \nabla _{g}u\right\vert
^{2}dv(g)+A_{\varepsilon }\int_{M}u^{2}dv(g)
\end{equation*}%
We test the Sobolev's inequality by $u_{m}-u,$ to get%
\begin{equation}
\left\Vert u_{m}-u^{+}\right\Vert _{N}^{2}\leq (1+\varepsilon )K_{\circ
}\int_{M}\left( \Delta _{g}(u_{m}-u^{+})\right) ^{2}dv(g)+o(1)\text{.}
\label{11}
\end{equation}%
Then (\ref{11}) implies that%
\begin{equation*}
\int_{M}f(x)\left\vert u_{m}-u^{+}\right\vert ^{N}dv(g)\leq (1+\varepsilon
)^{\frac{n}{n-4}}\max_{x\in M}f(x)K_{\circ }^{\frac{n}{n-4}}\left\Vert
\Delta _{g}(u_{m}-u^{+})\right\Vert _{2}^{N}+o(1)
\end{equation*}%
and by (\ref{10}) one writes%
\begin{equation*}
o(1)\geq \left\Vert \Delta _{g}\left( u_{m}-u^{+}\right) \right\Vert
_{2}^{2}-(1+\varepsilon )^{\frac{n}{n-4}}\max_{x\in M}f(x)K_{\circ }^{\frac{n%
}{n-4}}\left\Vert \Delta _{g}(u_{m}-u^{+})\right\Vert _{2}^{N}+o(1)\text{.}
\end{equation*}%
or in another words%
\begin{equation*}
o(1)\geq \left\Vert \Delta _{g}\left( u_{m}-u^{+}\right) \right\Vert
_{2}^{2}(1-(1+\varepsilon )^{\frac{n}{n-4}}\max_{x\in M}f(x)K_{\circ }^{%
\frac{n}{n-4}}\left\Vert \Delta _{g}\left( u_{m}-u^{+}\right) \right\Vert
_{2}^{N-2})+o(1)\text{.}
\end{equation*}%
Hence if%
\begin{equation*}
\limsup_{\newline
m\rightarrow +\infty }\left\Vert \Delta _{g}\left( u_{m}-u^{+}\right)
\right\Vert _{2}^{N-2}<\frac{1}{(1+\varepsilon )^{\frac{n}{n-4}}K_{\circ }^{%
\frac{n}{n-4}}\max_{x\in M}f(x)}
\end{equation*}%
we get%
\begin{equation*}
\frac{2}{n}\int_{M}\left( \Delta _{g}(u_{m}-u^{+})\right) ^{2}dv(g)<c.
\end{equation*}%
and since by assumption 
\begin{equation*}
c<\frac{2}{n\text{ }K_{\circ }^{\frac{n}{4}}(Max\text{ }_{x\in M}f(x))^{%
\frac{n-4}{4}}}
\end{equation*}%
we deduce that 
\begin{equation*}
\int_{M}\left( \Delta _{g}(u_{m}-u^{+})\right) ^{2}dv(g)<\frac{1}{K_{\circ
}^{\frac{n}{4}}(\max_{x\in M}f(x))^{\frac{n-4}{4}}}\text{.}
\end{equation*}%
Hence 
\begin{equation*}
o(1)\geq \left\Vert \Delta _{g}(u_{m}-u^{+})\right\Vert _{2}^{2}\underbrace{%
(1-(1+\varepsilon )^{\frac{n}{n-4}}\max_{x\in M}f(x)K_{\circ }^{\frac{n}{n-4}%
}\left\Vert \Delta _{g}(u_{m}-u^{+})\right\Vert _{2}^{N-2})}_{>0}+o(1)
\end{equation*}%
or%
\begin{equation*}
\left\Vert \Delta _{g}(u_{m}-u^{+})\right\Vert _{2}^{2}=o(1)
\end{equation*}%
i.e. $u_{m}\rightarrow u^{+}$ converges strongly in $H_{2}^{2}(M)$.\newline
Obviously $u^{+}\in N_{\lambda }$. We claim that $u^{+}\in N_{\lambda }^{+}$
since it is not the case $u^{+}\in N_{\lambda }^{-},$ thus $\left\langle
\nabla J_{\lambda }(u^{+}),u^{+}\right\rangle =0$ and $\left\langle \nabla
\Phi _{\lambda }(u^{+}),u^{+}\right\rangle <0,$ which implies that $%
J_{\lambda }(u^{+})>0,$ contradiction.\newline
Then,%
\begin{equation*}
J_{\lambda }(u^{+})=\alpha _{\lambda }^{+}=\alpha _{\lambda }<0\text{.}
\end{equation*}%
Now, we want to prove that $u^{+}$ is a trivial solution to equation (\ref%
{3'}) but this follows from Lemma\textbf{\ }\ref{lem1} since \ in that case $%
u^{+}$ is a global minimizer of $J_{\lambda }$ in $H_{2}^{2}(M).$
\end{proof}

\begin{theorem}
\label{thm3} Let $\lambda $ $\in (0,$ $\lambda _{\circ })$ and suppose that
a sequence $(u_{m})_{m}$ $\subset N_{\lambda }^{-}$ fulfils

\begin{equation*}
\left\{ 
\begin{array}{c}
J_{\lambda }(u_{_{m}})\leq c \\ 
\nabla J_{\lambda }(u_{_{m}})-\mu _{_{m}}\nabla \Phi _{\lambda
}(u_{m})\rightarrow 0%
\end{array}%
\right.
\end{equation*}%
.

with%
\begin{equation}
c<\frac{2}{n\text{ }K_{\circ }^{\frac{n}{4}}(Max\text{ }_{x\in M}f(x))^{%
\frac{n-4}{4}}}.  \label{C2}
\end{equation}%
Then the functional $J_{\lambda }$\ has a minimizer $u^{-}$ in $N_{\lambda
}^{-}$ and it satisfies

(i) $J_{\lambda }(u^{-})=\alpha _{\lambda }^{-}>0$,

(ii) $u^{-}$ is a nontrivial solution of equation (\ref{1}).
\end{theorem}

\begin{proof}
The proof is similar to that of Theorem \ref{thm2}, so we omit it.
\end{proof}

\begin{remark}
The nontrivial solutions $u^{+}$ and $u^{-}$ of equation (\ref{1}) given by
Theorem \ref{thm2} and Theorem\ref{thm3} are distinct since $u^{+}\in
N_{\lambda }^{+}$, $u^{-}\in N_{\lambda }^{-}$ and $N_{\lambda }^{+}\cap
N_{\lambda }^{-}=\emptyset $.
\end{remark}

\section{The sharp case $\protect\sigma =2$ and $\protect\mu =4$}

By section four, for any $\sigma \in (0,2)$ and $\mu \in (0,4)$, there is a
weak solution $u_{\sigma ,\mu }^{+}\in N_{\lambda }^{+}$ (resp. $u_{\sigma
,\mu }^{-}\in N_{\lambda }^{-}$) of equation (\ref{3'}). Now we are going to
show that the sequence $\left( u_{\sigma ,\mu }^{+}\right) _{\sigma ,\mu }$
and $\left( u_{\sigma ,\mu }^{-}\right) _{\sigma ,\mu }$are bounded in $%
H_{2}^{2}\left( M\right) $. First of all we have%
\begin{equation*}
J_{\lambda ,\sigma ,\mu }(u_{\sigma ,\mu })=\frac{1}{2}\left\Vert u_{\sigma
,\mu }\right\Vert ^{2}-\frac{1}{N}\int_{M}f(x)\left\vert u_{\sigma ,\mu
}\right\vert ^{N}dv_{g}-\frac{1}{q}\lambda \int_{M}\left\vert u_{\sigma ,\mu
}\right\vert ^{q}dv_{g}
\end{equation*}%
and since $u_{\sigma ,\mu }\in N_{\lambda }$, we infer that%
\begin{equation*}
J_{\lambda ,\sigma ,\mu }(u_{\sigma ,\mu })=\frac{N-2}{2N}\left\Vert
u_{\sigma ,\mu }\right\Vert ^{2}-\lambda \frac{N-q}{Nq}\int_{M}\left\vert
u_{\sigma ,\mu }\right\vert ^{q}dv_{g}\text{.}
\end{equation*}

For a smooth function $a$ on $M$, denotes by $a^{-}=\min \left( 0,\min_{x\in
M}(a(x)\right) $. Let $K(n,2,\sigma )$ the best constant and $A\left(
\varepsilon ,\sigma \right) $ the constants in the Hardy- Sobolev inequality.

Denote by $(u_{_{\sigma _{m},\mu _{m}}}^{+})_{m}$ \ a countable subsequence
of the sequence $(u_{_{\sigma ,\mu }}^{+})_{\sigma ,\mu }$ given above.

\begin{theorem}
Let $(M,g)$ be a Riemannian compact manifold of dimension $n\geq 5$. Let $%
\left( u_{m}^{+}\right) _{m}=(u_{_{\sigma _{m},\mu _{m}}}^{+})_{m}$ be a
sequence in $N_{\lambda }^{+}$ such that 
\begin{equation*}
\left\{ 
\begin{array}{c}
J_{\lambda ,\sigma ,\mu }(u_{_{m}}^{+})\leq c_{\sigma ,\mu } \\ 
\nabla J_{\lambda }(u_{_{m}}^{+})-\mu _{_{\sigma _{m},\mu _{m}}}\nabla \Phi
_{\lambda }(u_{_{m}}^{+})\rightarrow 0%
\end{array}%
\right. \text{.}
\end{equation*}%
Suppose that 
\begin{equation*}
c_{\sigma ,\mu }<\frac{2}{n\text{ }K\left( n,2\right) ^{\frac{n}{4}%
}(\max_{x\in M}f(x))^{\frac{n-4}{4}}}
\end{equation*}%
and 
\begin{equation*}
1+a^{-}\max \left( K(n,2,\sigma ),A\left( \varepsilon ,\sigma \right)
\right) +b^{-}\max \left( K(n,2,\mu ),A\left( \varepsilon ,\mu \right)
\right) >0
\end{equation*}%
then the equation%
\begin{equation}
\Delta ^{2}u-\nabla ^{\mu }(\frac{a}{\rho ^{2}}\nabla _{\mu }u)+\frac{bu}{%
\rho ^{4}}=f\left\vert u\right\vert ^{N-2}u+\lambda \left\vert u\right\vert
^{q-2}u  \label{12'}
\end{equation}%
has a non trivial solution $u^{+}\in N_{\lambda }^{+}$ in the distribution .
\end{theorem}

\begin{proof}
Let $\left( u_{m}^{+}\right) _{m}=\left( u_{\sigma _{m};\mu _{m}}^{+}\right)
_{m}\subset N_{\lambda ,\sigma ,\mu }^{+}$,%
\begin{equation*}
J_{\lambda ,\sigma ,\mu }(u_{m}^{+})=\frac{N-2}{2N}\left\Vert
u_{m}^{+}\right\Vert ^{2}-\lambda \frac{N-q}{Nq}\int_{M}\left\vert
u_{m}^{+}\right\vert ^{q}dv_{g}
\end{equation*}%
As in proof of Theorem \ref{thm2}, we get 
\begin{equation*}
J_{\lambda ,\sigma ,\mu }(u_{m}^{+})\geq
\end{equation*}%
\begin{equation*}
\left\Vert u_{m}^{+}\right\Vert ^{2}(\frac{N-2}{2N}-\lambda \frac{N-q}{Nq}%
\Lambda _{\sigma ,\mu }^{-\frac{q}{2}}V(M)^{1-\frac{q}{N}}(\max
((1+\varepsilon )K\left( n,2\right) ,A_{\varepsilon }))^{\frac{q}{2}}\tau
^{q-2})>0
\end{equation*}%
where $0<\lambda <\frac{\frac{\left( N-2\right) q}{2\left( N-q\right) }%
\Lambda _{\sigma ,\mu }^{\frac{q}{2}}}{V(M)^{1-\frac{q}{N}}(\max
((1+\varepsilon )K\left( n,2\right) ,A_{\varepsilon }))^{\frac{q}{2}}\tau
^{q-2}}=\lambda _{o,\sigma ,\mu }^{+}$ and $\Lambda _{\sigma ,\mu }$ is the
coercivity's constant ( which depends on $\sigma $ and $\mu )$

First we claim that%
\begin{equation*}
\lim_{\left( \sigma ,\mu \right) \rightarrow \left( 2^{-},4^{-}\right) }\inf
\Lambda _{\sigma ,\mu }>0\text{.}
\end{equation*}%
Indeed, if $\nu _{1,\sigma ,\mu }$ denotes the first nonzero eigenvalue of
the operator $u\rightarrow P_{g}(u)=\Delta _{g}^{2}u-div\left( \frac{a}{\rho
^{\sigma }}\nabla _{g}u\right) +\frac{bu}{\rho ^{\mu }}$, then clearly $%
\Lambda _{\sigma ,\mu }\geq \nu _{1,\sigma ,\mu }$. Suppose by absurd that $%
\lim_{\left( \sigma ,\mu \right) \rightarrow \left( 2^{-},4^{-}\right) }\inf
\Lambda _{\sigma ,\mu }=0$, then $\lim \inf_{\left( \sigma ,\mu \right)
\rightarrow \left( 2^{-},4^{-}\right) }\nu _{1,\sigma ,\mu }=0$.
Independently, if $u_{\sigma ,\mu }$ is the corresponding eigenfunction to $%
\nu _{1,\sigma ,\mu }$ we have%
\begin{equation*}
\nu _{1,\sigma ,\mu }=\left\Vert \Delta u\right\Vert _{2}^{2}+\int_{M}\frac{%
a\left\vert \nabla u\right\vert ^{2}}{\rho ^{\sigma }}dv_{g}+\int_{M}\frac{%
bu^{2}}{\rho ^{\mu }}dv_{g}
\end{equation*}%
\begin{equation}
\geq \left\Vert \Delta u\right\Vert _{2}^{2}+a^{-}\int \frac{\left\vert
\nabla u\right\vert ^{2}}{\rho ^{\sigma }}dv_{g}+b^{-}\int_{M}\frac{u^{2}}{%
\rho ^{\mu }}dv_{g}  \label{13}
\end{equation}%
where $a^{-}=\min \left( 0,\min_{x\in M}a(x)\right) $ and $b^{-}=\min \left(
0,\min_{x\in M}b(x)\right) $. The Hardy- Sobolev's inequality leads to%
\begin{equation*}
\int_{M}\frac{\left\vert \nabla u\right\vert ^{2}}{\rho ^{\sigma }}%
dv_{g}\leq C(\left\Vert \nabla \left\vert \nabla u\right\vert \right\Vert
^{2}+\left\Vert \nabla u\right\Vert ^{2})
\end{equation*}%
and since 
\begin{equation*}
\left\Vert \nabla \left\vert \nabla u\right\vert \right\Vert ^{2}\leq
\left\Vert \nabla ^{2}u\right\Vert ^{2}\leq \left\Vert \Delta u\right\Vert
^{2}+\beta \left\Vert \nabla u\right\Vert ^{2}
\end{equation*}%
where $\beta >0$ is a constant and it is well known that for any $%
\varepsilon >0$ there is a constant $c\left( \varepsilon \right) >0$ such
that 
\begin{equation*}
\left\Vert \nabla u\right\Vert ^{2}\leq \varepsilon \left\Vert \Delta
u\right\Vert ^{2}+c\left\Vert u\right\Vert ^{2}\text{.}
\end{equation*}%
Hence 
\begin{equation}
\int_{M}\frac{\left\vert \nabla u\right\vert ^{2}}{\rho ^{\sigma }}%
dv_{g}\leq C\left( 1+\varepsilon \right) \left\Vert \Delta u\right\Vert
^{2}+A\left( \varepsilon \right) \left\Vert u\right\Vert ^{2}  \label{14}
\end{equation}%
Now if $K(n,2,\sigma )$ denotes the best constant in inequality (\ref{14})
we get for any $\varepsilon >0$%
\begin{equation}
\int_{M}\frac{\left\vert \nabla u\right\vert ^{2}}{\rho ^{\sigma }}%
dv_{g}\leq \left( K(n,2,\sigma )^{2}+\varepsilon \right) \left\Vert \Delta
u\right\Vert ^{2}+A\left( \varepsilon ,\sigma \right) \left\Vert
u\right\Vert ^{2}\text{.}  \label{15}
\end{equation}%
By the inequalities (\ref{11}), (\ref{13}) and (\ref{15}), we have%
\begin{equation*}
\nu _{1,\sigma ,\mu }\geq \left( 1+a^{-}\max \left( K(n,2,\sigma ),A\left(
\varepsilon ,\sigma \right) \right) +b^{-}\max \left( K(n,2,\mu ),A\left(
\varepsilon ,\mu \right) \right) \right)
\end{equation*}%
\begin{equation*}
\times \left( \left\Vert \Delta u_{\sigma ,\mu }\right\Vert ^{2}+\left\Vert
u_{\sigma ,\mu }\right\Vert ^{2}\right) \text{.}
\end{equation*}%
So if 
\begin{equation*}
1+a^{-}\max \left( K(n,2,\sigma ),A\left( \varepsilon ,\sigma \right)
\right) +b^{-}\max \left( K(n,2,\mu ),A\left( \varepsilon ,\mu \right)
\right) >0
\end{equation*}%
then we get $\lim_{\sigma ,\mu }\left( u_{\sigma ,\mu }\right) =0$ and $%
\left\Vert u_{\sigma ,\mu }\right\Vert =1$ a contradiction.\ \ Denote by

\begin{equation*}
\Lambda =\lim \inf_{\sigma ,\mu }\newline
\Lambda _{\sigma ,\mu }\text{.}
\end{equation*}%
The same arguments as in the proof of Theorem\ref{thm2} we obtain that 
\begin{equation*}
a_{\lambda }^{+}\left( \frac{N-2}{Nq}-\frac{N-2}{2N}\right) ^{-1}\leq
\left\Vert u_{m}^{+}\right\Vert _{\sigma ,\mu }^{2}\leq -a_{\lambda
}^{+}\left( \frac{N-2}{Nq}-\frac{N-2}{2N}\right) ^{-1}+o(1)
\end{equation*}%
where 
\begin{equation*}
\left\Vert u^{+}\right\Vert _{\sigma ,\mu }^{2}=\left\Vert \Delta
u^{+}\right\Vert _{2}^{2}-\int_{M}\left( a(x)\frac{\left\vert \nabla
_{g}u^{+}\right\vert ^{2}}{\rho ^{\sigma }}+\frac{b(x)}{\rho ^{\mu }}\left(
u^{+}\right) ^{2}\right) dv(g)\text{.}
\end{equation*}%
It is easily seen by the Lebesque'dominated convergence theorem that $%
\left\Vert u^{+}\right\Vert _{\sigma ,\mu }$ goes to $\left\Vert
u^{+}\right\Vert _{2,4}$ as $\left( \sigma ,\mu \right) \rightarrow \left(
2,4\right) $.

Now by reflexivity of $H_{2}^{2}(M)$ and the compactness of the embedding $%
H_{2}^{2}(M)\subset H_{p}^{k}(M)$\ ( $k\ =0$,$1$; $p<N$), we obtain up to a
subsequence we have:

$u_{\sigma _{m},\mu _{m}}^{+}\rightarrow u^{+}$\ \ weakly in $H_{2}^{2}(M)$

$u_{\sigma _{m},\mu _{m}}^{+}\rightarrow u^{+}$ strongly in $L^{p}(M)$, $p<N$

$\nabla u_{\sigma _{m},\mu _{m}}^{+}\rightarrow \nabla u^{+}$ strongly in $%
L^{p}(M)$, $p<2^{\ast }=\frac{2n}{n-2}$

$u_{\sigma _{m},\mu _{m}}^{+}\rightarrow u^{+}$ a.e. in $M$.

For brevity we let $u_{m}^{+}=u_{\sigma _{m},\mu _{m}}^{+}$

The Br\'{e}zis-Lieb lemma allows us to write 
\begin{equation*}
\int_{M}\left( \Delta _{g}u_{m}^{+}\right) ^{2}dv_{g}=\int_{M}\left( \Delta
_{g}u^{+}\right) ^{2}dv_{g}+\int_{M}\left( \Delta
_{g}(u_{m}^{+}-u^{+})\right) ^{2}dv_{g}+o(1)
\end{equation*}%
and also 
\begin{equation*}
\int_{M}f(x)\left\vert u_{m}^{+}\right\vert
^{N}dv_{g}=\int_{M}f(x)\left\vert u^{+}\right\vert
^{N}dv_{g}+\int_{M}f(x)\left\vert u_{m}^{+}-u^{+}\right\vert ^{N}dv_{g}+o(1)%
\text{.}
\end{equation*}%
Now by the boundedness of the sequence ($u_{m}^{+})_{m}$, we have that $%
u_{m}^{+}\rightarrow u^{+}$\ \ weakly in $H_{2}^{2}(M)$, $\nabla
u_{m}^{+}\rightarrow \nabla u^{+}$\ \ weakly in $L^{2}(M,\rho ^{-2})$ and $%
u_{m}^{+}\rightarrow u^{+}$\ \ weakly in $L^{2}(M,\rho ^{-4})$\ i.e. for any 
$\varphi \in L^{2}(M)$If $\delta \in \left( 0,\delta (M)\right) $ then we
obtain for every $\varphi \in H_{2}^{2}\left( M\right) $ 
\begin{equation}
\int_{M}\frac{b(x)}{\rho ^{\mu _{m}}}\left( u_{m}^{+}\right) ^{2}\varphi
dv(g)=\int_{B_{P}(\delta )}\frac{b(x)}{\rho ^{\mu _{m}}}\left(
u_{m}^{+}\right) ^{2}dv(g)+\int_{M-B_{P}(\delta )}\frac{b(x)}{\rho ^{\mu
_{m}}}\left( u_{m}^{+}\right) ^{2}dv(g)  \label{11'}
\end{equation}%
and 
\begin{equation*}
\int_{M}\frac{b(x)}{\rho ^{\delta _{m}}}\left( u^{+}\right)
^{2}dv(g)=\int_{M}\frac{b(x)}{\rho ^{4}}\left( u^{+}\right) ^{2}dv(g)+o(1)\
\ \text{when}\ \ \delta _{m}\rightarrow 4^{-}\text{.}
\end{equation*}%
Now the fact $u_{m}^{+}\rightarrow u^{+}$\ \ weakly in $H_{2}^{2}(M)$, $%
\nabla u_{m}^{+}\rightarrow \nabla u^{+}$\ weakly in $L^{2}(M,\rho ^{-2})$
and$\ u_{m}^{+}\rightarrow u^{+}$ weakly in $L^{2}(M,\rho ^{-4})$\ expresses
as: for all $\varphi \in L^{2}(M):$%
\begin{equation*}
\int_{M}\frac{a(x)}{\rho ^{2}}\nabla u_{m}^{+}\nabla \varphi dv(g)=\int_{M}%
\frac{a(x)}{\rho ^{2}}\nabla u^{+}\nabla \varphi dv(g)+o(1)
\end{equation*}%
and%
\begin{equation*}
\int_{M}\frac{b(x)}{\rho ^{4}}u_{m}^{+}\varphi dv(g)=\int_{M}\frac{b(x)}{%
\rho ^{4}}u^{+}\varphi dv(g)+o(1).
\end{equation*}%
Consequently $u^{+}$ is a weak solution to equation (\ref{12'}).

Since $u_{m}^{+}\rightarrow u^{+}$weakly in $H_{2}^{2}(M)$, we have for all $%
\phi \in L^{2}(M)$%
\begin{equation*}
\int_{M}\left( u_{m}^{+}-u^{+}\right) \Delta _{g}^{2}\phi dv(g)=o(1)
\end{equation*}%
then,%
\begin{equation*}
\int_{M}u_{m}^{+}\Delta _{g}^{2}\phi dv(g)=\int_{M}\Delta _{g}\phi \Delta
_{g}u^{+}dv(g)+o(1)\text{.}
\end{equation*}%
For the second integral, we obtain%
\begin{equation*}
\int_{M}\left( \frac{a(x)}{\rho ^{\sigma _{m}}}\nabla _{g}u_{m}^{+}-\frac{%
a(x)}{\rho ^{2}}\nabla _{g}u^{+}\right) \nabla \phi dv(g)=
\end{equation*}%
\begin{equation*}
\int_{M}\left( \frac{a(x)}{\rho ^{\sigma _{m}}}\nabla _{g}u_{m}^{+}+\frac{%
a(x)}{\rho ^{2}}\left( \nabla _{g}u_{m}^{+}-\nabla _{g}u_{m}^{+}\right) -%
\frac{a(x)}{\rho ^{2}}\nabla _{g}u^{+}\right) \nabla \phi dv(g)\text{.}
\end{equation*}%
Consequently%
\begin{equation*}
\left\vert \int_{M}\left( \frac{a(x)}{\rho ^{\sigma _{m}}}\nabla
_{g}u_{m}^{+}-\frac{a(x)}{\rho ^{2}}\nabla _{g}u^{+}\right) \nabla \phi
dv(g)\right\vert \leq
\end{equation*}%
\begin{equation*}
\left\vert \int_{M}\left( \frac{a(x)}{\rho ^{\sigma _{m}}}\nabla
_{g}u_{m}^{+}-\frac{a(x)}{\rho ^{2}}\nabla _{g}u_{m}^{+}\right) \nabla \phi
dv(g)\right\vert +\left\vert \int_{M}\left( \frac{a(x)}{\rho ^{2}}\nabla
_{g}u_{m}^{+}-\frac{a(x)}{\rho ^{2}}\nabla _{g}u^{+}\right) \nabla \phi
dv(g)\right\vert
\end{equation*}%
\begin{equation*}
\leq \left\vert \int_{M}\frac{a(x)}{\rho ^{2}}\nabla _{g}\left(
u_{m}^{+}-u^{+}\right) \nabla \phi dv(g)\right\vert +\int_{M}\left\vert
a(x)\nabla \phi \nabla _{g}u_{m}^{+}\right\vert \left\vert \frac{1}{\rho
^{\sigma _{m}}}-\frac{1}{\rho ^{2}}\right\vert dv(g)\text{.}
\end{equation*}%
By the weak convergence in $L^{2}(M,\rho ^{-2})$ and the dominated
Lebesgue's convergence theorem, we obtain that 
\begin{equation*}
\int_{M}\left( \frac{a(x)}{\rho ^{\sigma _{m}}}\nabla _{g}u_{m}^{+}-\frac{%
a(x)}{\rho ^{2}}\nabla _{g}u^{+}\right) \nabla \phi dv(g)=o(1)\text{.}
\end{equation*}%
The third integral splits as%
\begin{equation*}
\int_{M}\left( \frac{b(x)}{\rho ^{\mu _{m}}}u_{m}^{+}-\frac{b(x)}{\rho ^{4}}%
u^{+}\right) \phi dv(g)=
\end{equation*}%
\begin{equation*}
\int_{M}\left( \frac{b(x)}{\rho ^{\mu _{m}}}u_{m}^{+}-\frac{b(x)}{\rho ^{4}}%
u_{m}^{+}+\frac{b(x)}{\rho ^{4}}u_{m}^{+}-\frac{b(x)}{\rho ^{4}}u^{+}\right)
\phi dv(g)
\end{equation*}%
so%
\begin{equation*}
\left\vert \int_{M}\left( \frac{b(x)}{\rho ^{\mu _{m}}}u_{m}^{+}-\frac{b(x)}{%
\rho ^{4}}u^{+}\right) \phi dv(g)\right\vert
\end{equation*}%
\begin{equation*}
\leq \int_{M}\left\vert b(x)\phi u_{m}\right\vert \left\vert \frac{1}{\rho
^{\mu _{m}}}-\frac{1}{\rho ^{4}}\right\vert dv(g)+\left\vert \int_{M}\frac{%
b(x)}{\rho ^{4}}\left( u_{m}^{+}-u^{+}\right) \phi dv(g)\right\vert
\end{equation*}%
and by the same arguments, we obtain that%
\begin{equation*}
\int_{M}\left( \frac{b(x)}{\rho ^{\delta _{m}}}u_{m}^{+}-\frac{b(x)}{\rho
^{4}}u^{+}\right) \phi dv(g)=o(1)\text{ \ \ .}
\end{equation*}

It remains to show that $\mu _{m}\rightarrow 0$ as $m$ $\rightarrow +\infty $
and $u_{m}^{+}\rightarrow u^{+}$ strongly in $H_{2}^{2}\left( M\right) $ but
this is the same as in the proof of Theorem \ref{thm3}.

Consequently $u^{+}$ is a nontrivial solution in $N_{\lambda }^{+}$ of
equation .\newline
\end{proof}

\section{Test Functions}

To give the proof of the main result, we consider a normal geodesic
coordinate system centred at $x_{o}$.\ Let $S_{x_{o}}(\rho )$ the geodesic
sphere centred at $x_{o}$ and of radius $\rho $ strictly less that the
injectivity radius $d$. Let $dv_{h}$ be the volume element of the $n-1$%
-dimensional Euclidean unit sphere $S^{n-1}$ endowed with its canonical
metric and put 
\begin{equation*}
G(\rho )=\frac{1}{\omega _{n-1}}\dint\limits_{S(\rho )}\sqrt{\left\vert
g(x)\right\vert }dv_{h}
\end{equation*}%
where $\omega _{n-1}$ is the volume of $S^{n-1}$ and $\left\vert
g(x)\right\vert $ the determinant of the Riemannian metric $g$. The Taylor's
expansion of $G(\rho )$ in a neighborhood of $x_{o}$ expresses as 
\begin{equation*}
G(\rho )=1-\frac{S_{g}(x_{\circ })}{6n}\rho ^{2}+o(\rho ^{2})
\end{equation*}%
where $S_{g}(x_{\circ })$ is the scalar curvature of $M$ at $x_{\circ }$.$%
\newline
$If $B(x_{\circ },\delta )$ is the geodesic ball centred at $x_{\circ }$ and
of radius $\delta $ such that $0<2\delta <d$, we consider the following
cutoff smooth function $\eta $ on $M$ 
\begin{equation*}
\eta (x)=\left\{ 
\begin{array}{c}
1\text{ \ \ \ \ \ \ \ \ on }B(x_{o},\delta ) \\ 
0\text{ \ \ on }M-B(x_{o},2\delta )%
\end{array}%
\right. \text{.}
\end{equation*}%
Define the following radial function

\begin{equation*}
u_{\epsilon }(x)=(\frac{(n-4)n(n^{2}-4)\epsilon ^{4}}{f(x_{\circ })})^{\frac{%
n-4}{8}}\frac{\eta (\rho )}{(\left( \rho \theta \right) ^{2}+\epsilon ^{2})^{%
\frac{n-4}{2}}}
\end{equation*}%
with%
\begin{equation*}
\theta =\left( 1+\left\Vert \frac{a}{\rho ^{\sigma }}\right\Vert
_{r}+\left\Vert \frac{b}{\rho ^{\mu }}\right\Vert _{s}\right) ^{\frac{1}{n}}
\end{equation*}%
where $\rho =d(x_{o},x)$ is the distance from $x_{o}$ to $x$ and $f(x_{\circ
})=\max_{x\in M}f(x)$. We need also the following integrals: for any real
positive numbers $p$, $q$ such that $p-q>1$ we put

\begin{equation*}
I_{p}^{q}=\int_{0}^{+\infty }\frac{t^{q}}{(1+t)^{p}}dt
\end{equation*}%
which fulfill the following relations 
\begin{equation*}
I_{p+1}^{q}=\frac{p-q-1}{p}I_{p}^{q}\text{ \ \ \ and\ \ \ \ \ }I_{p+1}^{q+1}=%
\frac{q+1}{p-q-1}I_{p+1}^{q}\text{.}
\end{equation*}

\section{Application to compact Riemannian manifolds of dimension $n>6$}

\begin{theorem}
\label{thm5} Let $\left( M,g\right) $ be a compact Riemannian manifold of
dimension $\ n>6$. Suppose that at a point $x_{o}$ where $f$ attains its
maximum the following condition%
\begin{equation*}
\frac{\Delta f(x_{o})}{f\left( x_{o}\right) }<\frac{1}{3}\left( \frac{%
(n-1)n\left( n^{2}+4n-20\right) }{\left( n^{2}-4\right) \left( n-4\right)
\left( n-6\right) }\frac{1}{\left( 1+\left\Vert \frac{a}{\rho ^{\sigma }}%
\right\Vert _{r}+\left\Vert \frac{b}{\rho ^{\mu }}\right\Vert _{s}\right) ^{%
\frac{4}{n}}}-1\right) S_{g}\left( x_{o}\right)
\end{equation*}%
holds . Then the equation (\ref{1}) has at least two non trivial solutions.
\end{theorem}

\begin{proof}
The proof of Theorem \ref{Th1} reduces to show that the condition (\ref{C1})
of Theorem \ref{thm2} \ which is the same condition (\ref{C2}) of Theorem %
\ref{ref2} is satisfied and since at the end of section 1, we have shown
that for a given $u\in H_{2}^{2}\left( M\right) $ there exist two real
numbers $t^{-}>0$ and $t^{+}>0$ such that $t^{-}u\in N_{\lambda }^{-}$ and $%
t^{+}u\in N_{\lambda }^{+}$ \ for sufficiently small $\lambda $, so it
suffices to show that%
\begin{equation*}
\sup_{t>0}J_{\lambda }\left( tu_{\epsilon }\right) <\frac{1}{K_{\circ }^{%
\frac{n}{4}}\left( \max_{x\in M}f(x)\right) ^{\frac{n}{4}-1}}\text{.}
\end{equation*}%
The expression of $\int_{M}f(x)\left\vert u_{\epsilon }(x)\right\vert
^{N}dv_{g}$ is well known (see for example \cite{10} ) and is given in case $%
n>6$ by%
\begin{equation*}
\int_{M}f(x)\left\vert u_{\epsilon }(x)\right\vert ^{N}dv_{g}=\frac{\theta
^{-n}}{K_{\circ }^{\frac{n}{4}}(f(x_{\circ }))^{\frac{n-4}{4}}}\left( 1-(%
\frac{\Delta f(x_{\circ })}{2(n-2)f(x_{\circ })}+\frac{S_{g}(x_{\circ })}{%
6(n-2)})\epsilon ^{2}+o(\epsilon ^{2})\right) \text{.}
\end{equation*}%
The following estimation is computed in \cite{7} and is given by $\ $%
\begin{equation*}
\int_{M}\frac{a(x)}{\rho ^{\sigma }}\left\vert \nabla u_{\epsilon
}\right\vert ^{2}dv_{g}\leq
\end{equation*}%
\begin{equation*}
2^{-1+\frac{1}{r}}\theta ^{-n\frac{r}{r-1}}(n-4)^{2}\left( \frac{%
(n-4)n(n^{2}-4)\epsilon ^{4}}{f(x_{\circ })}\right) ^{\frac{n-4}{4}%
}\left\Vert \frac{a}{\rho ^{\sigma }}\right\Vert _{r}\omega _{n-1}^{1-\frac{1%
}{r}}\epsilon ^{-\left( n-4\right) +2-\frac{n}{r}}
\end{equation*}%
\begin{equation*}
\times \left( I_{\frac{\left( n-2\right) r}{r-1}}^{1+\frac{n-2}{2}.\frac{r-1%
}{r}}+o(\epsilon ^{2})\right) \text{.}
\end{equation*}%
Letting 
\begin{equation}
A=K_{\circ }^{\frac{n}{4}}\frac{(n-4)^{\frac{n}{4}+1}\times \left( \omega
_{n-1}\right) ^{\frac{r-1}{r}}}{2^{\frac{r-1}{r}}}(n(n^{2}-4))^{\frac{n-4}{4}%
}\left( I_{\frac{\left( n-2\right) r}{r-1}}^{\frac{n-2}{2}+\frac{r}{r-1}%
}\right) ^{\frac{r-1}{r}}  \label{20}
\end{equation}%
we obtain 
\begin{equation*}
\int_{M}a(x)\left\vert \nabla u_{\epsilon }\right\vert ^{2}dv_{g}\leq
\epsilon ^{2-\frac{n}{r}}\theta ^{-n\frac{r}{r-1}}\frac{A}{\text{ }K_{\circ
}^{\frac{n}{4}}(f(x_{\circ }))^{\frac{n-4}{4}}}\left\Vert \frac{a}{\rho
^{\sigma }}\right\Vert _{r}\left( 1+o(\epsilon ^{2})\right) \text{.}
\end{equation*}%
Also the estimation of the third term of $J_{\lambda }$ is computed in \cite%
{7} as 
\begin{equation*}
\int_{M}\frac{b(x)}{\rho ^{\mu }}u_{\epsilon }^{2}dv_{g}\leq \left\Vert
b\right\Vert _{s}\left( \frac{(n-4)n(n^{2}-4)}{f(x_{\circ })}\right) ^{\frac{%
n-4}{4}}(\frac{\omega _{n-1}}{2})^{\frac{s-1}{s}}\epsilon ^{4-\frac{n}{s}%
}\theta ^{-n\frac{s}{s-1}}
\end{equation*}%
\begin{equation*}
\times \left( \left( I_{\frac{\left( n-4\right) s}{(s-1)}}^{\frac{n}{2}%
}\right) ^{\frac{s-1}{s}}+o(\epsilon ^{2})\right)
\end{equation*}%
Putting 
\begin{equation}
B=K_{\circ }^{\frac{n}{4}}((n-4)n(n^{2}-4))^{\frac{n-4}{4}}(\frac{\omega
_{n-1}}{2})^{\frac{s-1}{s}}\left( I_{\frac{\left( n-4\right) s}{(s-1)}}^{%
\frac{n}{2}}\right) ^{\frac{s-1}{s}}  \label{21}
\end{equation}%
we get%
\begin{equation*}
\int_{M}b(x)u_{\epsilon }^{2}dv_{g}\leq \epsilon ^{4-\frac{n}{s}}\theta ^{-n%
\frac{s}{s-1}}\frac{\left\Vert \frac{b}{\rho ^{\mu }}\right\Vert _{s}B}{%
\text{ }K_{\circ }^{\frac{n}{4}}(f(x_{\circ }))^{\frac{n-4}{4}}}\left(
1+o(\epsilon ^{2})\right) \text{.}
\end{equation*}%
The computation of $\int_{M}\left( \Delta u_{\epsilon }\right) ^{2}dv_{g}$
is well known see for example (\cite{10})\ and is given by%
\begin{equation*}
\int_{M}\left( \Delta u_{\epsilon }\right) ^{2}dv_{g}=\frac{\theta ^{-n}}{%
K_{\circ }^{\frac{n}{4}}(f(x_{\circ }))^{\frac{n-4}{4}}}\left( 1-\frac{%
n^{2}+4n-20}{6(n^{2}-4)(n-6)}S_{g}(x_{\circ })\epsilon ^{2}+o(\epsilon
^{2})\right) \text{.}
\end{equation*}%
Resuming we get%
\begin{equation*}
\int_{M}\left( \Delta u_{\epsilon }\right) ^{2}-a(x)\left\vert \nabla
u_{\epsilon }\right\vert ^{2}+b(x)u_{\epsilon }^{2}dv_{g}\leq \frac{\theta
^{-n}}{K_{\circ }^{\frac{n}{4}}f(x_{\circ })^{\frac{n-4}{4}}}\times
\end{equation*}%
\begin{equation*}
\left( 1+\epsilon ^{2-\frac{n}{r}}\theta ^{-\frac{n}{r-1}}A\left\Vert \frac{a%
}{\rho ^{\sigma }}\right\Vert _{r}+\epsilon ^{4-\frac{n}{s}}\theta ^{-\frac{n%
}{s-1}}B\left\Vert \frac{b}{\rho ^{\mu }}\right\Vert _{s}-\frac{n^{2}+4n-20}{%
6(n^{2}-4)(n-6)}S_{g}(x_{\circ })\epsilon ^{2}+o(\epsilon ^{2})\right) \text{%
.}
\end{equation*}%
Now, we have 
\begin{equation*}
J_{\lambda }\left( tu_{\epsilon }\right) \leq J_{o}\left( tu_{\epsilon
}\right) =\frac{t^{2}}{2}\left\Vert u_{\epsilon }\right\Vert ^{2}-\frac{t^{N}%
}{N}\int_{M}f(x)\left\vert u_{\epsilon }(x)\right\vert ^{N}dv_{g}
\end{equation*}%
\begin{equation*}
\leq \frac{\theta ^{-n}}{K_{\circ }^{\frac{n}{4}}f(x_{\circ })^{\frac{n-4}{4}%
}}\left\{ \frac{1}{2}t^{2}\left( 1+\epsilon ^{2-\frac{n}{r}}\theta ^{-\frac{n%
}{r-1}}A\left\Vert \frac{a}{\rho ^{\sigma }}\right\Vert _{r}+\epsilon ^{4-%
\frac{n}{s}}\theta ^{-\frac{n}{s-1}}B\left\Vert \frac{b}{\rho ^{\mu }}%
\right\Vert _{s}\right) -\frac{t^{N}}{N}\right.
\end{equation*}%
\begin{equation*}
\left. +\left[ \left( \frac{\Delta f(x_{o})}{2\left( n-2\right) f(x_{o})}+%
\frac{S_{g}\left( x_{o}\right) }{6\left( n-1\right) }\right) \frac{t^{N}}{N}-%
\frac{1}{2}t^{2}\frac{n^{2}+4n-20}{6\left( n^{2}-4\right) \left( n-6\right) }%
S_{g}\left( x_{o}\right) \right] \epsilon ^{2}\right\}
\end{equation*}%
\begin{equation*}
+o\left( \epsilon ^{2}\right)
\end{equation*}%
and letting $\epsilon $ small enough so that%
\begin{equation*}
1+\epsilon ^{2-\frac{n}{r}}\theta ^{-\frac{n}{r-1}}A\left\Vert \frac{a}{\rho
^{\sigma }}\right\Vert _{r}+\epsilon ^{4-\frac{n}{s}}\theta ^{-\frac{n}{s-1}%
}B\left\Vert \frac{b}{\rho ^{\mu }}\right\Vert _{s}\leq \left( 1+\left\Vert 
\frac{a}{\rho ^{\sigma }}\right\Vert _{r}+\left\Vert \frac{b}{\rho ^{\mu }}%
\right\Vert _{s}\right) ^{\frac{4}{n}}
\end{equation*}%
and since the function $\varphi (t)=\alpha \frac{t^{2}}{2}-\frac{t^{N}}{N}$,
with $\alpha >0$ and $t>0$, attains its maximum at $t_{o}=\alpha ^{\frac{1}{%
N-2}}$ and%
\begin{equation*}
\varphi (t_{o})=\frac{2}{n}\alpha ^{\frac{n}{4}}\text{.}
\end{equation*}%
Consequently, we get%
\begin{equation*}
J_{\lambda }\left( tu_{\epsilon }\right) \leq \frac{2\theta ^{-n}}{nK_{\circ
}^{\frac{n}{4}}f(x_{\circ })^{\frac{n-4}{4}}}\left\{ 1+\left\Vert \frac{a}{%
\rho ^{\sigma }}\right\Vert _{r}+\left\Vert \frac{b}{\rho ^{\mu }}%
\right\Vert _{s}\right.
\end{equation*}%
\begin{equation*}
\left. +\left[ \left( \frac{\Delta f(x_{o})}{2\left( n-2\right) f(x_{o})}+%
\frac{S_{g}\left( x_{o}\right) }{6\left( n-1\right) }\right) \frac{t_{o}^{N}%
}{N}-\frac{1}{2}t_{o}^{2}\frac{n^{2}+4n-20}{6\left( n^{2}-4\right) \left(
n-6\right) }S_{g}\left( x_{o}\right) \right] \epsilon ^{2}\right\}
\end{equation*}%
\begin{equation*}
+o\left( \epsilon ^{2}\right) .
\end{equation*}

Taking account of the value of $\theta $ and putting%
\begin{equation*}
R(t)=\left( \frac{\Delta f(x_{o})}{2\left( n-2\right) f(x_{o})}+\frac{%
S_{g}\left( x_{o}\right) }{6\left( n-1\right) }\right) \frac{t^{N}}{N}-\frac{%
1}{2}\frac{n^{2}+4n-20}{6\left( n^{2}-4\right) \left( n-6\right) }%
S_{g}\left( x_{o}\right) t^{2}
\end{equation*}%
we obtain 
\begin{equation*}
\sup_{t\geq 0}J_{\lambda }\left( tu_{\epsilon }\right) <\frac{2}{nK_{\circ
}^{\frac{n}{4}}\left( \max_{x\in M}f(x)\right) ^{\frac{n}{4}-1}}
\end{equation*}%
provided that $R(t_{o})<0$ i.e.%
\begin{equation*}
\frac{\Delta f(x_{o})}{f\left( x_{o}\right) }<\left( \frac{n\left(
n^{2}+4n-20\right) }{3\left( n+2\right) \left( n-4\right) \left( n-6\right) }%
\frac{1}{\left( 1+\left\Vert \frac{a}{\rho ^{\sigma }}\right\Vert
_{r}+\left\Vert \frac{b}{\rho ^{\mu }}\right\Vert _{s}\right) ^{\frac{4}{n}}}%
-\frac{n-2}{3\left( n-1\right) }\right) S_{g}\left( x_{o}\right) \text{.}
\end{equation*}%
Which completes the proof.
\end{proof}

\section{Application to compact Riemannian manifolds of dimension $n=6$}

\begin{theorem}
In case $n=6$, we suppose that at a point $x_{o}$ where $f$ attains its
maximum $S_{g}\left( x_{o}\right) >0$. Then the equation (\ref{1}) has at
least two distinct non trivial solutions in the distribution sense..
\end{theorem}

\begin{proof}
In case $n=6$ the only term whose expression differs from the case $n>6$ is
the first term of $J_{\lambda }$ and is given ( see for example \cite{10}) by

\begin{equation*}
\int_{M}\left( \Delta u_{\epsilon }\right) ^{2}dv_{g}=\frac{\theta ^{n}}{%
K_{\circ }^{\frac{n}{4}}(f(x_{\circ }))^{\frac{n-4}{4}}}\left( 1-\frac{%
2\left( n-4\right) }{n^{2}(n^{2}-4)I_{n}^{\frac{n}{2}-1}}S_{g}(x_{\circ
})\epsilon ^{2}\log \left( \frac{1}{\epsilon ^{2}}\right) +O(\epsilon
^{2})\right) \text{.}
\end{equation*}%
Letting $\epsilon $ so that 
\begin{equation*}
1+\epsilon ^{2-\frac{n}{r}}\theta ^{-\frac{n}{r-1}}A\left\Vert \frac{a}{\rho
^{\sigma }}\right\Vert _{r}+\epsilon ^{4-\frac{n}{s}}\theta ^{-\frac{n}{s-1}%
}B\left\Vert \frac{b}{\rho ^{\mu }}\right\Vert _{s}\leq \left( 1+\left\Vert 
\frac{a}{\rho ^{\sigma }}\right\Vert _{r}+\left\Vert \frac{b}{\rho ^{\mu }}%
\right\Vert _{s}\right) ^{\frac{4}{n}}
\end{equation*}%
where $A$ and $B$ are given by (\ref{20}) and (\ref{21}), we get%
\begin{equation*}
J_{\lambda }\left( u_{\epsilon }\right) \leq \frac{1}{2}\left\Vert
u_{\epsilon }\right\Vert ^{2}-\frac{1}{N}\int_{M}f(x)\left\vert u_{\epsilon
}(x)\right\vert ^{N}dv_{g}
\end{equation*}%
\begin{equation*}
\leq \frac{\theta ^{n}}{\text{ }K_{\circ }^{\frac{n}{4}}(f(x_{\circ }))^{%
\frac{n-4}{4}}}\left[ \frac{t^{2}}{2}\left( 1+\left\Vert \frac{a}{\rho
^{\sigma }}\right\Vert _{r}+\left\Vert \frac{b}{\rho ^{\mu }}\right\Vert
_{s}\right) ^{1-\frac{4}{n}}-\frac{t^{N}}{N}\right.
\end{equation*}%
\begin{equation*}
\left. -\frac{n-4}{n^{2}\left( n^{2}-4\right) I_{n}^{\frac{n}{2}-1}}\theta
^{-2}S_{g}(x_{\circ })t^{2}\epsilon ^{2}\log \left( \frac{1}{\epsilon ^{2}}%
\right) \right] +O(\epsilon ^{2})\text{.}
\end{equation*}%
As in the case $n>6$ we infer that 
\begin{equation*}
\max_{t\geq 0}J_{\lambda }\left( tu_{\epsilon }\right) <\frac{2}{n\text{ }%
K_{\circ }^{\frac{n}{4}}(f(x_{\circ }))^{\frac{n-4}{4}}}
\end{equation*}%
provided that%
\begin{equation*}
S_{g}(x_{\circ })>0\text{.}
\end{equation*}%
Which achieves the proof.
\end{proof}

\end{document}